\documentclass[12pt,a4paper,oneside]{amsart}
\usepackage[a4paper]{geometry}
\usepackage{mathptmx} 
\DeclareMathAlphabet{\mathcal}{OMS}{cmsy}{m}{n} 

\usepackage{amssymb,enumerate,mathrsfs}

\usepackage{perpage}

\MakePerPage[2]{footnote}

\newtheorem{theorem}{Theorem}[section]
\newtheorem{proposition}[theorem]{Proposition}
\newtheorem{lemma}[theorem]{Lemma}
\newtheorem{corollary}[theorem]{Corollary}

\theoremstyle{definition}
\newtheorem{remark}[theorem]{Remark}

\numberwithin{equation}{section}

\newcommand{\vertbar}{\>|\>}
\newcommand{\set}[2]{\ensuremath{\{ #1 \vertbar #2 \}}}

\DeclareMathOperator{\ad}{ad}
\DeclareMathOperator{\C}{C}
\DeclareMathOperator{\Cent}{Cent}
\DeclareMathOperator{\dcobound}{d}
\DeclareMathOperator{\E}{E}
\DeclareMathOperator{\Der}{Der}
\DeclareMathOperator{\Hom}{Hom}
\DeclareMathOperator{\Homol}{H}
\DeclareMathOperator{\HC}{HC}
\DeclareMathOperator{\Har}{Har}
\DeclareMathOperator{\Ker}{Ker}
\DeclareMathOperator{\ess}{\mathsf{S}}
\DeclareMathOperator{\sym}{S}
\DeclareMathOperator{\Sl}{\mathsf{sl}}
\DeclareMathOperator{\TR}{TR}
\DeclareMathOperator{\Z}{Z}

\begin{document}

\title{
Deformations of current Lie algebras. I. Small algebras in characteristic $2$}

\author{Alexander Grishkov}
\address[Alexander Grishkov]{
Institute of Mathematics and Statistics, University of S\~{a}o Paulo,
S\~{a}o Paulo, Brazil and Omsk F.M. Dostoevsky State University, Omsk, Russia
}
\email{shuragri@gmail.com}

\author{Pasha Zusmanovich}
\address[Pasha Zusmanovich]{
Department of Mathematics, University of Ostrava, Ostrava, Czech Republic
}
\email{pasha.zusmanovich@osu.cz}

\date{last minor revision July 7, 2017}
\thanks{J. Algebra \textbf{473} (2017), 513--544; arXiv:1410.3645}

\begin{abstract}
We compute low-degree cohomology of current Lie algebras extended over 
the $3$-di\-men\-si\-o\-nal simple algebra, compute deformations of related 
semisimple Lie algebras, and apply these results to classification of simple Lie
algebras of absolute toral rank $2$ and having a Cartan subalgebra of toral rank
one. Everything is in characteristic $2$.
\end{abstract}

\maketitle

\section*{Introduction}

This is the first in the planned series of papers, devoted to a unified
approach to computation of deformations of current Lie algebras and algebras
close to them, and applications of those deformations in the structure theory. 
This first paper is devoted to some particular characteristic $2$ case.

Throughout the paper, the ground field $K$ is assumed to be of characteristic 
$2$ (sometimes with additional qualifications, such as being perfect or 
algebraically closed), unless stated otherwise. When referring to other results
in modular Lie algebras theory, we customary refer to ``big characteristics'', 
what should mean ``characteristic $p > 2$'', or, depending on the context, 
``characteristic $p \ne 2$'', unless specified otherwise.

The classification of finite-dimensional simple Lie algebras over an 
algebraically closed field of characteristic $p>3$ was completed relatively
recently (see \cite{strade-intro}). While there was a significant progress in 
understanding of an unwieldy zoo of existing examples of simple Lie algebras in
characteristic $2$ and $3$ in a series of recent papers by Dimitry Leites and 
his collaborators (see \cite{bgl} and \cite{blls} and references therein), the 
classification problem in these characteristics remain widely open.

Our starting point is another remarkable paper -- by Serge Skryabin 
\cite{skryabin} -- in which he proved, among other things, that in 
characteristic $2$, there are no simple Lie algebras of absolute toral rank $1$,
and characterized simple Lie algebras having a Cartan subalgebra of toral rank 
$1$ as certain filtered deformations of semisimple Lie algebras having the socle
of the form $S \otimes \mathcal O$, where $S$ belongs to a certain family of 
simple Lie algebras, and $\mathcal O$ is a divided powers algebra.

In this paper we compute these deformations for the case where $S$ is the 
smallest algebra in the family: the $3$-dimensional simple Lie algebra. In doing
so, we follow the standard nowadays approach by Murray Gerstenhaber (for a nice
overview, see, for example, \cite[\S 8]{gs}) in which infinitesimal deformations
are described by the second cohomology of the underlying Lie algebra with 
coefficients in the adjoint module, and obstructions for prolongations of 
infinitesimal deformations live in the third cohomology. Even this very 
particular, at the first glance, case is of considerable interest: it allows to
complete the ongoing classification of simple Lie algebras of absolute toral 
rank $2$, due to the first author and Alexander Premet. (In \cite{skryabin2}
it is cautiously conjectured that any restricted simple Lie algebra of absolute
toral rank $2$ is isomorphic either to $A_2$, or to $G_2$.)

There are similarities between this paper and \cite{me-deformations}: in both 
cases, driven by structure theory of modular Lie algebras (classification of
simple Lie algebras with a solvable maximal subalgebra in big characteristics in
\cite{me-deformations}, and classification of certain simple Lie algebras in 
characteristic $2$ in the present paper), we compute some filtered deformations
of semisimple Lie algebras, naturally appearing in the corresponding 
classification problems (the present paper is related to \cite{skryabin} in more
or less the same way as \cite{me-deformations} is related to \cite{weisfeiler}).
In fact, we do a bit more: we compute low-degree cohomology and deformations of the corresponding current Lie algebra 
$S \otimes A$, as well as its extensions by derivations, where instead of 
divided powers algebra we take an arbitrary associative commutative algebra $A$.
These intermediate computations lead to interesting formulae intertwining 
various cohomological invariants of $S$ and $A$, complement investigations in 
\cite{me-deformations}, \cite{low} and \cite{without-unit}, and are of 
independent interest.

In the subsequent papers we plan to extend these computations to all the cases 
appearing in Skryabin's theorem, i.e., where $S$ is either Zassenhaus or 
Hamiltonian algebra\footnote{
In characteristic $2$, unlike in big characteristics, there are two types of
simple Hamiltonian algebras, alternate and nonalternate, depending on the type 
of the differential $2$-form the algebra preserves; see \cite[\S 1.1.2]{blls} 
and references therein. The Hamiltonian algebras we are dealing with are of 
alternate type (denoted by $\mathfrak h_{\Pi}(n, \underline{N})$ in 
\cite{blls}).}, 
and to obtain in this way a full description of simple Lie algebras having a 
Cartan subalgebra of toral rank $1$, what should be an important intermediate 
step in classification efforts. However, the more or less direct computational 
approach of this paper will meet considerable difficulties if we will try to 
extend it to that generality. For example, as it will be clear from the 
discussion below, the second cohomology $\Homol^2(S,S)$ is involved, and the 
latter cohomology in the case of general Zassenhaus algebra in characteristic 
$2$ seems to be enormous (some relevant computations were made in 1980s by Askar
Dzhumadil'daev, but no full account is available in the literature; and the case
of Hamiltonian algebras appears to be even more cumbersome). Thus, different, 
more subtle, approaches will be needed.

The contents of the paper are as follows. In the preliminary \S \ref{sec-prelim}
we fix notation and recall the necessary notions (current algebras, various
cohomology theories, etc.). In \S\S \ref{sec-h2-zass} and \ref{sec-h2-adj} we 
obtain results about the second cohomology with the coefficients in the trivial 
and adjoint module, respectively, of the current Lie algebra extended over the 
$3$-dimensional simple Lie algebra. In \S \ref{sec-plus-d} we glue these results
together to compute the positive part of the second cohomology, and filtered 
deformations of the corresponding semisimple Lie algebra. In \S \ref{sec-fin} we
apply the preceding results to derive the main result of the paper: there are no 
``new'' simple Lie algebras of absolute toral rank $2$ and with Cartan 
subalgebra of toral rank $1$. This result will be used in the forthcoming 
classification of simple Lie algebras of absolute toral rank $2$. The last
\S \ref{sec-15} contains a brief discussion of a family of $15$-dimensional
simple Lie algebras appearing during the proof.

\section{Preliminaries}\label{sec-prelim}

\subsection{Zassenhaus algebras}

Recall a construction of the ubiquitous Zassenhaus algebras (see, for example, 
\cite[Vol. I, \S 7.6]{strade-intro}).

Let $A$ be an associative commutative algebra with unit, and $D$ a derivation of
$A$. Then the set of derivations $AD = \set{aD}{a \in D}$ is a Lie algebra of 
derivations of $A$. Assuming that this Lie algebra is also a free 
(one-dimensional) $A$-module, and identifying $aD$ with $a$, $AD$ can be 
considered as a Lie algebra structure on $A$ with the bracket 
$[a,b] = aD(b) - bD(a)$ for any $a,b\in A$. In characteristic $2$, the latter 
formula is equivalent to 
\begin{equation}\label{eq-brack}
[a,b] = D(ab) .
\end{equation}

(Note that algebras with multiplication (\ref{eq-brack}) were considered also
in big characteristics, but, of course, then they are no longer Lie algebras.
They belong to the class of so-called Novikov--Jordan algebras -- commutative 
algebras satisfying a certain identity of degree $4$, see 
\cite{dzhu-novikov-jord}. In characteristic $2$ the classes of Lie algebras
and Novikov--Jordan algebras have nontrivial intersection -- for example, all
algebras with multiplication (\ref{eq-brack}) -- but these classes do not 
coincide).

Recall that the divided powers algebra $\mathcal O_1(n)$ over a field of 
characteristic $p$ is defined as a $p^n$-dimensional algebra having a basis 
$\{x^{(i)}\}$, $0 \le i < p^n$, with multiplication 
$x^{(i)} x^{(j)} = \binom{i+j}j x^{(i+j)}$. The special derivation $\partial$ of
$\mathcal O_1(n)$ is defined as
$$
\partial(x^{(i)}) = \begin{cases} 
x^{(i-1)} & \text{ if } i>0 \\
0         & \text{ if } i=0.
\end{cases}
$$

Specializing the bracket (\ref{eq-brack}) to the case $A = \mathcal O_1(n)$,
and $D = \partial$, we arrive at the Lie algebra $W_1(n)$ of dimension $2^n$
with the basis $\{e_i = x^{(i+1)} \partial | -1 \le i \le 2^n-2 \}$ and 
multiplication
$$
[e_i,e_j] = \begin{cases}
\binom{i+j+2}{i+1} e_{i+j} &\text{if } -1 \le i+j \le 2^n-2 \\
0                          &\text{otherwise}.
\end{cases}
$$
(The reader may be puzzled for a second by the unusual coefficient 
$\binom{i+j+2}{i+1}$ instead of the usual one $\binom{i+j+1}j - \binom{i+j+1}i$,
but it is obvious that in characteristic $2$ the latter is equal to the former; 
we adopt the usual convention that $\binom{i}{j} = 0$ if $i<j$.)

The algebra $W_1(n)$ is a subalgebra of the whole derivation algebra 
$\Der(\mathcal O_1(n))$, the latter is freely generated as an
$\mathcal O_1(n)$-module by $\partial, \partial^p, \dots, \partial^{p^{n-1}}$.

In characteristic $2$, unlike in big characteristics, the algebra $W_1(n)$ is 
not simple, but its commutant $W_1^\prime(n)$ of dimension $2^n-1$, linearly 
spanned by elements $\{e_i | -1 \le i \le 2^n-3 \}$, is. By abuse of terminology, we will refer to 
both algebras $W_1(n)$ and $W_1^\prime(n)$ as \emph{Zassenhaus algebras}. 
(In \cite{skryabin} the latter algebra is denoted by $K_1^\prime(n)$.)

In the first nontrivial case $n=2$, the algebra $W_1^\prime(2)$ is 
$3$-dimensional. This is the case we will mainly deal with, so we adopt a special
notation for this algebra, $\ess$, and its basic elements: 
$e = e_{-1}$, $h = e_0$, and $f = e_1$. The algebra $\ess$ has multiplication
table 
\begin{equation}\label{eq-mult}
[e,h] = e, \quad [f,h] = f, \quad [e,f] = h ,
\end{equation}
and is an analog of $\Sl(2)$ in big characteristics. Further, denoting $g=e_2$,
we get a basis $\{e,h,f,g\}$ of $W_1(2)$ where, in addition to (\ref{eq-mult}),
$$
[e,g] = f, \quad [h,g] = 0, \quad [f,g] = 0.
$$

\subsection{Current and semisimple algebras}

Given a Lie algebra $L$ and associative commutative algebra $A$, the Lie algebra
$L \otimes A$ with the bracket
$$
[x \otimes a, y \otimes b] = [x,y] \otimes ab
$$ 
for any $x,y \in L$, $a,b \in A$, is referred to as \emph{current Lie algebra}.
If $A$ contains the unit $1$, then there is an obvious embedding of Lie algebras
$L \hookrightarrow L \otimes A$ induced by the map $x \mapsto x \otimes 1$, 
$x\in L$.

We can extend current Lie algebras by derivations. For example, if $\mathfrak D$
is a Lie algebra of outer derivations of $S$, $\mathfrak E$ is a Lie algebra of
derivations of $A$, and $U$ is an $\mathfrak E$-invariant subspace of $A$, we 
can consider an extension of the current algebra $S \otimes A$ of the form 
\begin{equation}\label{eq-semisimple}
S \otimes A + \mathfrak D \otimes U + \mathfrak E ,
\end{equation}
(here, and in the subsequent similar constructions, ``+'' refers to the 
semidirect sum, while the sign $\oplus$ is reserved for the direct sum of vector spaces), where $\mathfrak D \otimes U$ and $\mathfrak E$ act on 
$S \otimes A$ via the first and the second tensor factor, respectively:
\begin{align*}
[x \otimes a, D \otimes u] &= D(x) \otimes au \\
[x \otimes a, E] &= x \otimes E(a)
\end{align*}
for any $x \in S$, $a \in A$, $u\in U$, $D \in \mathfrak D$, 
$E \in \mathfrak E$, and the Lie bracket between $\mathfrak D \otimes U$ and $\mathfrak E$ is defined
by taking their commutator as maps acting on $S \otimes A$:
$$
[D \otimes u, E] = D \otimes E(u) .
$$
(Of course, more complex extensions are possible, where derivations are 
``mixed'', i.e., involve non-splittable sums of terms from $\Der(S) \otimes B$
and $\Der(A)$, but the ``homogeneous'' case above will be enough for our 
purposes here.)

The significance of such constructions in the structure theory of modular Lie 
algebras stems from the fact that, according to the classical Block theorem,
every finite-dimensional semisimple Lie algebra over a field of positive 
characteristic is sandwiched between the direct sum of current Lie algebras
of the form $S \otimes \mathcal O$, where $S$ is a simple Lie algebra, and 
$\mathcal O$ is a divided powers algebra, and its Lie algebra of derivations, 
i.e., the direct sum of algebras of the form 
$\Der(S) \otimes \mathcal O + \Der(\mathcal O)$ (see 
\cite[Vol. I, Corollary 3.3.5]{strade-intro}). In particular, it happens often in the structure theory that description of some classes of Lie algebras 
reduces to elucidation of the structure of some classes of deformations of semisimple Lie algebras of the kind (\ref{eq-semisimple}),
or similar algebras. The present paper is an instance of such elucidation.

\subsection{Gradings}\label{subseq-grad}

The basic elements of the Zassenhaus algebra provide the gradings
$$
W_1^\prime(n) = \bigoplus_{i=-1}^{2^n-3} Ke_i \quad\text{and}\quad
W_1(n) = \bigoplus_{i=-1}^{2^n-2} Ke_i ,
$$ 
which will be referred to as \emph{standard gradings}.

Any grading $S = \bigoplus_i S_i$ on an algebra $S$ induces a grading on the 
current algebra $S \otimes A = \bigoplus_i (S_i \otimes A)$. In the cases of 
$S = \ess$ and $S = W_1(2)$, this induced grading on the respective current 
algebra will be also referred to as \emph{standard}.

The extended algebras of the form (\ref{eq-semisimple}) acquire the grading
induced from the standard grading on $S \otimes A$, assigning the respective
weights to elements of $\mathfrak D \otimes U$ according to weights of the
external derivations in $\mathfrak D$, and putting $\mathfrak E$ to the zero
component.

Recall that for $\mathbb Z$-gradings, \emph{depth} is the minimal index of a 
negative nonzero component, and \emph{length} is the maximal index of a positive
nonzero component. For example, the standard gradings on $\ess$ and $W_1(2)$
have length $1$ and $2$, respectively, and both of them have depth $1$.

\subsection{Lie algebra cohomology}
Given a Lie algebra $L$ and an $L$-module $M$, the corresponding $n$th 
cohomology will be denoted by $\Homol^n(L,M)$ (we will be concerned exclusively
with the cases $n=1,2,3$, and the trivial module $K$ or the adjoint module $L$).
In evaluating cohomology, we will use repeatedly the following fact. If 
$L = \bigoplus_{\alpha \in G} L_\alpha$ is a grading of a Lie algebra $L$ by an
additively written abelian group $G$, and 
$M = \bigoplus_{\alpha \in G} M_\alpha$ is a $G$-graded $L$-module (in the case
of the trivial module, all the grading is concentrated in degree $0$), then the
Chevalley--Eilenberg complex computing $\Homol^n(L,M)$ decomposes into the 
direct sum of complexes:
\begin{equation}\label{eq-shift}
\Homol^n(L,M) \simeq \bigoplus_{\lambda \in G} \Homol^n_\lambda(L,M) ,
\end{equation}
where $\Homol^n_\lambda(L,M)$ is linearly spanned by the classes of cocycles 
$\varphi$ satisfying the condition 
$$
\varphi(L_{\alpha_1}, \dots, L_{\alpha_n}) \subseteq 
M_{\alpha_1 + \dots + \alpha_n - \lambda}
$$ 
for any $\alpha_1, \dots, \alpha_n, \lambda \in G$. We will call $n$ and
$\lambda$ the \emph{cohomology degree} and \emph{weight}, respectively.

In the case $G = \mathbb Z$, the \emph{positive cohomology} is defined as
$$
\Homol^n_+(L,M) = \bigoplus_{\lambda > 0} \Homol^n_\lambda(L,M) .
$$

When speaking about positive cohomology of algebras $\ess \otimes A$ and their
extended algebras of the form (\ref{eq-semisimple}), we always mean the standard
grading.

When specifying cocycles (and, more generally, cochains) on algebras by giving 
the cocycle values on all appropriate combinations of basic elements, we will omit 
the zero values. For example, if we say (as, for example, in case (\ref{en-1}) of Proposition 
\ref{prop-h2-triv} below) that a cocycle 
$\Phi: (\ess \otimes A) \wedge (\ess \otimes A) \to K$ is given by 
\begin{equation*}
(f \otimes a) \wedge (f \otimes b) \mapsto \xi(ab) ,
\end{equation*} 
that implicitly assumes that $\Phi$ vanishes on
$(e \otimes A) \wedge (e \otimes A)$,  
$(e \otimes A) \wedge (f \otimes A)$,  
$(e \otimes A) \wedge (h \otimes A)$,  
$(h \otimes A) \wedge (h \otimes A)$, and
$(h \otimes A) \wedge (f \otimes A)$.

\subsection{Deformations}
The significance of the positive cohomology stems from the fact that it is 
responsible for description of filtered deformations of the graded algebra $L$, 
i.e., filtered Lie algebras such that their associated graded algebra is 
isomorphic to $L$. Infinitesimal deformations lie in $\Homol^2_+(L,L)$, and 
obstructions to their prolongability are described by the Massey brackets, 
defined as cohomology classes in 
$\Homol^3_+(L,L)$ of a cocycle $[[\varphi,\psi]] + [[\psi,\varphi]]$, where
\begin{equation*}
[[\varphi,\psi]](x,y,z) = 
\varphi(\psi(x,y),z) + \varphi(\psi(z,x),y) + \varphi(\psi(y,z),x)
\end{equation*}
for any $x,y,z \in L$. (For details, see \cite{gs} or \cite[\S 2.1]{bgl}. The 
latter reference addresses the peculiarities of characteristic $2$. Note that, 
due to these peculiarities, the definition of Massey brackets adopted here deviates from the
standard one: it accounts for only a ``half'' of the usual terms and is not 
symmetric.)

As in the case of positive cohomology, when speaking about filtered deformations
of algebras, we always mean the standard grading.

\subsection{Cyclic cohomology}
One of the beauties of cohomology of current Lie algebras $L \otimes A$ is that
it intertwines various cohomology theories of the underlying algebras $L$ and 
$A$. We will need two such theories -- or, rather, merely their low-degree 
incarnations -- associated with the algebra $A$.

Recall that a map $f$ of two variables is \emph{symmetric} if it satisfies the
condition $f(a,b) = f(b,a)$ for any its two arguments $a,b$, and is 
\emph{alternating} if $f(a,a) = 0$. Obviously, in the class of bilinear maps, 
alternating maps are symmetric, but not vice versa.

Consider bilinear maps $\alpha: A \times A \to K$ satisfying
the cocycle equation
\begin{equation}\label{eq-coc-a}
\alpha(ab,c) + \alpha(ca,b) + \alpha(bc,a) = 0
\end{equation}
for any $a,b,c\in A$. Such \emph{symmetric} maps form the first cyclic cohomology $\HC^1(A)$ (see, for
example, \cite[Proposition 2.1.14]{loday} for equivalence of this definition and
the usual definition in terms of the double complex; note also that normally we
would say here ``skew-symmetric'', but in characteristic $2$ skew-symmetry and
symmetry are the same). Note an obvious but useful fact: if $A$ contains a unit
$1$, then $\alpha(1,A) = 0$ for any $\alpha \in \HC^1(A)$. 

We will consider a variation of the first cyclic cohomology: the space of 
\emph{alternating} maps satisfying the same cocycle equation (\ref{eq-coc-a}) 
will be denoted by $\widehat{\HC^1}(A)$. We have an obvious inclusion 
$\widehat{\HC^1}(A) \subseteq \HC^1(A)$.

\subsection{Harrison cohomology}
$\Har^n(A,A)$ denotes the $n$th degree Harrison cohomology of $A$. See 
\cite{harrison} where this cohomology is introduced, and \cite{gs} for a more 
modern treatment, but all what we will need is an interpretation of Harrison 
cohomology in low degrees: Harrison $1$-coboundaries vanish, Harrison 
$1$-cocycles are the same as Hochschild $1$-cocycles, i.e., derivations, 
so $\Har^1(A,A)$ coincides with $\Der(A)$, the space (actually, a Lie algebra)
of derivations of $A$; Harrison $2$-cochains form the space of symmetric bilinear
maps $A \times A \to A$, denoted by $\sym^2(A,A)$;
Harrison $2$-coboundaries are the same as Hochschild $2$-coboundaries, and 
Harrison $2$-cocycles are symmetric Hochschild $2$-cocycles. 
An easy but useful observation is that every Harrison $2$-cochain $\alpha$ is
cohomologous to one with $\alpha(1,A) = 0$ (\cite[p.~194]{harrison}).

Another elementary observation which, nevertheless, will play role in the 
sequel: any Harrison $1$-cocycle, i.e., derivation, of an algebra $A$, vanishes
on $A^{[2]}$, where $A^{[2]}$ is the subalgebra of $A$ linearly spanned by 
squares of all elements (obviously, already the \emph{set} of all squares forms
a subring in $A$, and is a subalgebra if the ground field is perfect).

\subsection{Toral ranks, Skryabin's theorem}
Recall that a (relative) \emph{toral rank} of a subalgebra $S$ of a Lie algebra 
$L$ is the maximal dimension of tori in the $p$-subalgebra generated by
$S$ in the $p$-envelope of $L$. An \emph{absolute toral rank} $\TR(L)$ 
of a Lie algebra $L$ is the toral rank of $L$ as a subalgebra of itself (in 
other words, the maximal dimension of tori in the $p$-envelope of $L$). In the 
classification scheme of simple Lie algebras, both in small and big 
characteristics, a significant role is played by intermediate classifications of
simple Lie algebras having Cartan subalgebras of ``small'' toral rank.

We quote the relevant part of \cite[Theorem 6.3]{skryabin}:
\begin{theorem}[Skryabin]\label{th-skr}
A simple finite-dimensional Lie algebra over an algebraically closed field, 
having a Cartan subalgebra of toral rank $1$, is isomorphic either to the 
Zassenhaus algebra $W_1^\prime(n)$ ($n > 1$), or to the Hamiltonian algebra in 
two variables $H_2^{\prime\prime}(n,m)$ ($n,m > 1$), or to a filtered
deformation of the graded semisimple Lie algebra $L$ such that
\begin{equation}\label{eq-form}
S \otimes \mathcal O_1(n) \subset L \subseteq 
\Der(S) \otimes \mathcal O_1(n) + K\partial ,
\end{equation}
where either $n=2$ and $S \simeq W_1^\prime(n)$, or $n=1$ and 
$S \simeq H_2^{\prime\prime}(n,m)$. The grading on $L$ is of depth $1$ and the 
induced grading on the socle $S \otimes \mathcal O_1(n)$ is standard.
\end{theorem}

Note that the condition of the grading to be standard, implicit in the 
formulation of the theorem in \cite{skryabin}, follows from Theorems 1.1 and 1.2
of that paper.

Concrete realization of Hamiltonian algebras is immaterial for our purpose here.
All what we need to know is the absolute toral rank of the algebras involved:
$\TR(W_1^\prime(n)) = n$ and $\TR(H_2^{\prime\prime}(n,m)) = n+m-1$.

We are going to describe filtered deformations appearing in Theorem \ref{th-skr}
in the simplest case $S = \ess$.

\section{
Second cohomology of $\ess \otimes A$ with trivial coefficients
}\label{sec-h2-zass}

As explained in the introduction, we are interested in low-degree cohomology
of $\ess \otimes A$. Our first goal is to compute $\Homol^2(\ess \otimes A, K)$.
(In what follows, $A$ denotes an arbitrary associative commutative algebra with
unit.)

There are general formulas for $\Homol^2(L \otimes A,K)$ for an arbitrary Lie 
algebra $L$ (see \cite[Theorem 1]{without-unit}), but they are valid in 
characteristic $\ne 2,3$. To extend these results to the case of 
characteristic $2$, new notions and techniques will be needed. We defer their 
development to a subsequent paper, and treat here the case $L = \ess$ via
straightforward computations.

\begin{proposition}\label{prop-h2-triv}
$$
\Homol^2(\ess \otimes A,K) \simeq 
(A/A^{[2]})^* \oplus (A/A^{[2]})^* \oplus \widehat{\HC^1}(A) .
$$
The basic cocycles can be chosen as follows:
\begin{enumerate}[\upshape(i)]
\item\label{en-1}
$(f \otimes a) \wedge (f \otimes b) \mapsto \xi(ab)$;
\item\label{en-2}
$(e \otimes a) \wedge (e \otimes b) \mapsto \xi(ab)$,
where, in both (\ref{en-1}) and (\ref{en-2}), 
$\xi: A \to K$ is a linear map such that $\xi(A^{[2]}) = 0$;
\item\label{en-3}
$(e \otimes a) \wedge (f \otimes b) \mapsto \alpha(a,b)$ \\
$(h \otimes a) \wedge (h \otimes b) \mapsto \alpha(a,b)$,
where $\alpha \in \widehat{\HC^1}(A)$.
\end{enumerate}
\end{proposition}

\begin{proof}
Consider the standard $\mathbb Z$-grading of $\ess \otimes A$, and the induced 
decomposition (\ref{eq-shift}) of its cohomology:
\begin{equation*}
\Homol^2(\ess \otimes A, K) = 
\bigoplus_{k\in \mathbb Z} \Homol^2_k(\ess \otimes A, K) .
\end{equation*}
Obviously, the nonzero terms here are possible only for $-2 \le k \le 2$.

Further, the root space decomposition of $\ess \otimes A$ with respect to the
toral element $h \otimes 1$ has the form 
\begin{equation}\label{eq-root-sp}
\ess \otimes A = L_{\overline 0} \oplus L_{\overline 1} ,
\end{equation}
where $L_{\overline 0} = L_0 = h \otimes A$, and 
$L_{\overline 1} = L_{-1} \oplus L_1 = \langle e, f \rangle \otimes A$. Hence, 
by the theorem of invariance of cohomology with respect to the torus action
(see, for example, \cite[Chapter 1, \S 5.2]{fuchs}), we may restrict our 
attention to the subcomplex generated by cocycles respecting this decomposition.
Being coupled with decomposition (\ref{eq-root-sp}), this shows that it will be
enough to consider cocycles of weight $-2$, $0$ and $2$.

The proof of Proposition \ref{prop-h2-triv} now follows from the two Lemmas 
below, and noting that the case of cohomology of weight $2$ is, by symmetry
interchanging $e$ and $f$, completely similar to those of weight $-2$ described
in Lemma \ref{lem-0}.
\end{proof}

\begin{lemma}\label{lem-0}
$\Homol^2_{-2}(\ess \otimes A, K) \simeq (A/A^{[2]})^*$. The basic cocycles can
be chosen as in part (\ref{en-1}) of Proposition \ref{prop-h2-triv}.
\end{lemma}

\begin{proof}
Let $\Phi$ be a $2$-cocycle with trivial coefficients on $\ess \otimes A$ 
of weight $-2$. We may write 
$$
\Phi(f \otimes a, f \otimes b) = \alpha(a,b)
$$
for any $a,b \in A$, and for some bilinear alternating map 
$\alpha: A \times A \to K$. Writing the cocycle equation for the triple 
$h \otimes a, f \otimes b, f \otimes c$, we get 
$\alpha(ab,c) + \alpha(ac,b) = 0$, which implies $\alpha(a,b) = \xi(ab)$ for 
some linear map $\xi: A \to K$. Since $\alpha$ is alternating, we have 
$\xi(A^{[2]}) = 0$. As there are no nonzero linear maps 
$\Omega: \ess \otimes A \to K$, and, hence, coboundaries, of weight $-2$, the 
linear independence of cocycles of weight $-2$ implies their cohomological 
independence.
\end{proof}

\begin{lemma}\label{lemma-h2-triv-0}
$\Homol^2_0(\ess \otimes A, K) \simeq \widehat{\HC^1}(A)$. The basic cocycles 
can be chosen as in part (\ref{en-3}) of Proposition \ref{prop-h2-triv}.
\end{lemma}

\begin{proof}
Let $\Phi$ be a $2$-cocycle with trivial coefficients on $\ess \otimes A$ 
of weight $0$. We may write 
\begin{align*}
\Phi(e \otimes a, f \otimes b) &= \alpha(a,b) \\
\Phi(h \otimes a, h \otimes b) &= \beta(a,b)
\end{align*}
for any $a,b\in A$, and for some bilinear maps 
$\alpha, \beta: A \times A \to K$. Since $\Phi$ is alternating, $\beta$ is 
alternating too.

Any coboundary of weight $0$ is generated by a linear map 
$\Omega: \ess \otimes A \to K$ defined by 
$\Omega(e \otimes a) = \Omega(f \otimes a) = 0$ and
$\Omega(h \otimes a) = \omega(a)$ for some linear map $\omega: A \to K$, and 
hence has the form $\dcobound\Phi(e \otimes a, f \otimes b) = \omega(ab)$, with
zero values on all other pairs of basic elements. Setting 
$\omega(a) = \alpha(a,1)$, and modifying the cocycle $\Phi$ by the respective
coboundary, we may assume $\Phi(e \otimes a, f \otimes 1) = 0$, i.e., 
$\alpha(a,1) = 0$.

Writing the cocycle equation for the triple 
$e \otimes a, h \otimes b, f \otimes c$, we get 
$$
\alpha(ab,c) + \beta(ac,b) + \alpha(bc,a) = 0
$$
for any $a,b,c \in A$, which implies $\beta = \alpha$ and 
$\alpha \in \widehat{\HC^1}(A)$, thus arriving at cocycles of type (\ref{en-3})
in Proposition \ref{prop-h2-triv}.

The equality $\Phi = \dcobound \Omega$ for some linear map 
$\Omega : \ess \otimes A \to K$ of weight $0$, yields $\alpha = 0$, which shows
that these cocycles are cohomologically dependent if and only if the 
corresponding $\alpha$'s are linearly dependent.
\end{proof}

\begin{corollary}\label{cor-h2-triv-+}
$\Homol^2_+(\ess \otimes A, K) \simeq (A/A^{[2]})^*$. The basic cocycles can be
chosen as in part (\ref{en-2}) of Proposition \ref{prop-h2-triv}.
\end{corollary}

\begin{proof}
We have $\Homol^2_+(\ess \otimes A, K) = \Homol^2_2(\ess \otimes A, K)$.
\end{proof}

\section{
Low-degree cohomology of $\ess \otimes A$ with adjoint coefficients
}\label{sec-h2-adj}

Now let us turn to cohomology of $\ess \otimes A$ with coefficients in the 
adjoint module. A few general remarks are in order.

As noted in \cite[Lemma 1.1]{vavilov-fest}, for any Lie algebra $L$
and non-negative integer $k$, the space $\Homol^k(L \otimes A, L \otimes A)$ 
contains $\Homol^k(L,L) \otimes A$. On the level of cocycles, this embedding
is given by 
\begin{equation}\label{eq-coc-1}
(x_1 \otimes a_1) \wedge \dots \wedge (x_k \otimes a_k) \mapsto 
\varphi(x_1, \dots, x_k) \otimes a_1 \cdots a_k u
\end{equation}
for any $x_1, \dots, x_k \in L$, $a_1, \dots, a_k \in A$, where $\varphi$ is
some $k$-cocycle $\varphi$ on $L$, and $u\in A$.

Similarly, one can show that $\Homol^k(L \otimes A, L \otimes A)$ always 
contains $\Cent(L) \otimes \Har^k(A,A)$. Here $\Cent(L)$ is the centroid of $L$,
i.e., the space of linear maps $\omega: L \to L$ such that 
$\omega([x,y]) = [x,\omega(y)]$ for any $x,y\in L$. The latter embedding is, 
however, more involved -- due to the fact that higher Harrison cohomology is 
determined in terms of complicated permutations and is sensitive to 
characteristic of the ground field -- and we restrict ourselves to the cases of
low cohomology degree (as that is all what we need here anyway): on the level of $1$- and $2$-cocycles, the asserted embedding is given by
\begin{equation}\label{eq-coc-2}
x \otimes a \mapsto \omega(x) \otimes d(a)
\end{equation}
and
\begin{equation}\label{eq-coc-3}
(x \otimes a) \wedge (y \otimes b) \mapsto \omega([x,y]) \otimes \alpha(a,b) 
\end{equation}
respectively, for any $x,y\in L$, $a,b\in A$, and where $\omega \in \Cent(L)$,
$d\in \Der(A) = \Har^1(A,A)$, $\alpha \in \Har^2(A,A)$. 

Moreover, the cocycles of types (\ref{eq-coc-1}) and 
(\ref{eq-coc-2})-(\ref{eq-coc-3}) are cohomologically independent, so 
$\Homol^k(L \otimes A, L \otimes A)$ always contains the direct sum 
\begin{equation}\label{eq-summand}
\Big(\Homol^k(L,L) \otimes A\Big) \oplus 
\Big(\Cent(L) \otimes \Har^k(A,A)\Big) .
\end{equation}
For an (easy straightforward) proof, see \cite[Theorem 2.1]{low} ($k=1$) and \cite[\S 2]{me-deformations} ($k=2$).

We will prove that in the case of interest that is all what we have:

\begin{proposition}\label{prop}
For $n=1,2$, we have
$\Homol^n(\ess\otimes A, \ess\otimes A) \simeq 
\Big(\Homol^n(\ess,\ess) \otimes A\Big) \oplus \Har^n(A,A)$.
\end{proposition}

\begin{remark}\hfill

{\bf 1)}
As $\ess$ is, obviously, central, i.e., its centroid coincides with the ground
field, the second summand in (\ref{eq-summand}) is reduced to a mere 
$\Har^n(A,A)$.

{\bf 2)}
This is similar with the cases of characteristics $0$ and $>3$. 
In these cases, for $n=1,2$,
\begin{equation}\label{eq-gen}
\Homol^n(\Sl(2) \otimes A, \Sl(2) \otimes A) \simeq \Har^n(A,A) .
\end{equation}
The case $n=1$ is a particular case of many previous results in the literature,
of which \cite[Theorem 2.1]{low} is, perhaps, the most general one (see the 
proof of Proposition \ref{prop-h1} below). The case $n=2$ is a particular case 
of \cite{C} (where similar formulas are obtained for all finite-dimensional simple Lie algebras over an 
algebraically closed field of characteristic zero), and of \cite{bennis} (where
cohomology up to degree $4$ is computed). The proofs there utilize complete reducibility
of representations of $\Sl(2)$ and other facts peculiar to characteristic
zero case.

In the case of characteristic $>3$, the formula (\ref{eq-gen}) is established
in \cite[Proposition 2.8]{me-deformations}. Our computations are organized 
similar to those in \cite{me-deformations}, but characteristic $2$ case turns 
out to be more cumbersome. The reason for this is twofold. First, unlike the
cases of big characteristic, $\Homol^2(\ess,\ess)$ is not zero (see 
Lemma \ref{lemma-ess} below). Second, the root space decomposition 
(\ref{eq-root-sp}) with respect to the toral element $h \otimes 1$ is more 
coarse than in the big characteristic case (the roots $-1$ and $1$ coalesce).
\end{remark}

The case $n=1$ of Proposition \ref{prop} is a particular case of the general
statement:

\begin{proposition}\label{prop-h1}
If $L$ is perfect and central Lie algebra, then 
\begin{equation*}
\Homol^1(L \otimes A, L \otimes A) \simeq 
\Big(\Homol^1(L,L) \otimes A\Big) \oplus \Der(A) .
\end{equation*}
\end{proposition}

\begin{proof}
This is a direct consequence of \cite[Corollary 2.2]{low}. The standing
assumption of \cite{low} is characteristic $\ne 2,3$ of the ground field, but an
easy inspection of the proof of Theorem 2.1 from that paper, on which 
Corollary 2.2 is based, shows that it is characteristic-free. As $L$ 
is perfect and central, the term in the general formula for 
$\Homol^1(L\otimes A, L\otimes A)$ involving commutant vanishes, and the 
centroid of $L$ coincides with the ground field, so we are left with the desired
isomorphism.
\end{proof}

We start the proof of the $n=2$ case with recording necessary facts about 
cohomology of $\ess$.

\begin{lemma}\label{lemma-ess}\hfill
\begin{enumerate}[\upshape(i)]
\item\label{en-ess-1}
$\Homol^1(\ess, \ess)$ is $2$-dimensional, with basic cocycles (outer 
derivations) given by
$(\ad e)^2$ and $(\ad f)^2$.  
\item
$\Homol^2(\ess, \ess)$ is $2$-dimensional, with basic cocycles given by
\begin{equation}\label{eq-coc-s}
f \wedge h \mapsto e\phantom{ .}
\end{equation}
and
\begin{equation}\label{eq-coc-s-2}
e \wedge h \mapsto f .
\end{equation}
\end{enumerate}
\end{lemma}

\begin{proof}
Direct computations, similar to those performed below, but simpler.
\end{proof}

\begin{remark}\label{rem-lemma-ess}\hfill

{\bf 1)}
This, essentially, provides information on the whole cohomology of $\ess$ with 
coefficients in the adjoint module, as $\Homol^0(\ess,\ess)$ coincides with the
center of $\ess$ and hence is zero, and 
$\Homol^3(\ess,\ess) \simeq \Homol_0(\ess,\ess) = 0$ by the Poincar\'e duality.
(Here $\Homol_0$ denotes the $0$th \emph{homology}, not to be confused with the
subscript denoting the \emph{cohomology} weight. As this is the only place in 
the paper where homology makes -- a rather superficial -- appearance, this 
should not lead to confusion.)

{\bf 2)}
Part (i) is contained implicitly in \cite{pet}, where the $5$-dimensional 
algebra $\Der(\ess)$ is used to exhibit some phenomenon peculiar to 
characteristic $2$. Part (ii) is stated in \cite[Lemma 4.1]{bgl}.

{\bf 3)}
The algebra linearly spanned by $\ess$ and $(\ad f)^2$, is isomorphic to 
$W_1(2)$ (with $g$ being identified with $(\ad f)^2$).

{\bf 4)}
Any linear combination of the nontrivial $2$-cocycles exhibited in (ii) has 
trivial Massey bracket, and hence prolonged trivially to a (nontrivial) 
deformation of $\ess$. However, any such deformation is isomorphic to $\ess$.
Such phenomenon, under the name of \emph{semitrivial deformations}, is 
thoroughly investigated in \cite{blls}.
\end{remark}

\begin{proof}[Proof of Proposition \ref{prop}, case $n=2$]
Employing again the standard grading and root space decomposition, like in the 
proof of Proposition \ref{prop-h2-triv}, we obtain
\begin{equation*}
\Homol^2(\ess \otimes A, \ess \otimes A) = 
\bigoplus_{k\in \mathbb Z} \Homol^2_k(\ess \otimes A, \ess \otimes A) .
\end{equation*}
with nonzero terms for $k=-2,0,2$.

The next lemma is formulated in a setting slightly more general than it is 
needed for computation of cohomology of $\ess \otimes A$ of weight $2$. This 
will be used later in Proposition \ref{prop-h2-d}, where we will compute the 
positive cohomology of an extended algebra of the form (\ref{eq-semisimple}).

\begin{lemma}\label{lemma-minus-2-ext}
Let $\xi: A \to K$ and $D: A \to A$ be linear maps such that $\xi(1) = 0$ and 
$D(1) = 0$, and $\Lambda: A \times A \to A$ a bilinear alternating map such that
$\Lambda(1,A) = 0$. Then there exists a cochain 
$\Phi \in \C^2_2(\ess \otimes A, \ess \otimes A)$ such that its coboundary 
$\dcobound \Phi$ has the form:
\begin{align}\label{eq-cond}
\begin{split}
(e \otimes a) \wedge (e \otimes b) &\wedge (e \otimes c) \mapsto 
e \otimes \Big(\xi(ab)D(c) + \xi(ca)D(b) + \xi(bc)D(a)\Big)        
\\
(e \otimes a) \wedge (e \otimes b) &\wedge (h \otimes c) \mapsto 
h \otimes \xi(ab)D(c)
\\
(e \otimes a) \wedge (e \otimes b) &\wedge (f \otimes c) \mapsto   
f \otimes \Big(\xi(ab)D(c) + a\Lambda(b,c) + b\Lambda(a,c)\Big)
\\
(e \otimes a) \wedge (h \otimes b) &\wedge (h \otimes c) \mapsto   
f \otimes a\Lambda(b,c)
\end{split} 
\end{align}
if and only if 
\begin{equation}\label{eq-yoyo3}
\Lambda(a,b) = \xi(a)D(b) + \xi(b)D(a)
\end{equation}
and
\begin{multline}\label{eq-yoyo5}
\Big(\xi(a)b + \xi(b)a + \xi(ab)1\Big)D(c) +
\Big(\xi(c)a + \xi(a)c + \xi(ca)1\Big)D(b) +
\Big(\xi(b)c + \xi(c)b + \xi(bc)1\Big)D(a) \\ = 0
\end{multline}
for any $a,b,c \in A$. 

Moreover, in this case $\Phi$ is equal to the sum of a $2$-coboundary and a map
of the form
\begin{align*}
\begin{split}
(e \otimes a) \wedge (e \otimes b) &\mapsto h \otimes 
\Big(\xi(a)D(b) + \xi(b)D(a)\Big)
\\
(e \otimes a) \wedge (h \otimes b) &\mapsto f \otimes \Big(abv + \xi(b)D(a)\Big)
\end{split}
\end{align*}
for some $v\in A$. 
\end{lemma}

\begin{proof}
We may write
\begin{gather*}
\Phi(e \otimes a, e \otimes b) = h \otimes \beta(a,b)   \\
\Phi(e \otimes a, h \otimes b) = f \otimes \alpha(a,b)
\end{gather*}
for any $a,b \in A$, and for some bilinear maps 
$\alpha, \beta: A \times A \to A$, where $\beta$ is alternating.

Modifying $\Phi$ by the coboundary $\dcobound\Omega$, where the cochain
$\Omega: \ess \otimes A \to \ess \otimes A$ is defined as follows:
$\Omega (e \otimes a) = f \otimes \beta(1,a)$, and 
$\Omega (f \otimes A) = \Omega (h \otimes A) = 0$, we may assume that, up to 
coboundaries, $\Phi (e \otimes 1, e \otimes a) = 0$, i.e., $\beta(1,a) = 0$.

Evaluating $\dcobound \Phi$ for all $4$ triples present in the condition 
(\ref{eq-cond}), we get respectively:
\begin{gather}
\beta(a,b)c + \beta(c,a)b + \beta(b,c)a = 
\xi(ab)D(c) + \xi(ca)D(b) + \xi(bc)D(a) \label{eq-fff}
\\
\beta(ac,b) + \beta(bc,a) + b\alpha(a,c) + a\alpha(b,c) = \xi(ab)D(c) 
\label{eq-ffh} \\
\alpha(b,ac) + \alpha(a,bc) + c\beta(a,b) = 
\xi(ab)D(c) + a\Lambda(b,c) + b\Lambda(a,c) \label{eq-one} 
\\
\alpha(ab,c) + \alpha(ac,b) + c\alpha(a,b) + b\alpha(a,c) = a\Lambda(b,c)
\label{eq-two}
\end{gather}
for any $a,b,c \in A$. (Since $\Phi$ is of weight $2$, the cocycle equation 
$\dcobound\Phi = 0$ for other basic triples is satisfied automatically.)

Substituting $c=1$ in (\ref{eq-fff}), we get
\begin{equation}\label{eq-yoyo4}
\beta(a,b) = \xi(a)D(b) + \xi(b)D(a)
\end{equation}
for any $a,b \in A$. Substituting $a=1$ in (\ref{eq-one}), we get
$$
\alpha(b,c) = \alpha(1,bc) + \xi(b)D(c) + \Lambda(b,c)
$$
for any $b,c \in A$. Substituting the latter equality in (\ref{eq-two}) with 
$a=1$, we get
\begin{equation}\label{eq-three}
b\alpha(1,c) + c\alpha(1,b) = \xi(b)D(c) + \xi(c)D(b) + \Lambda(b,c)
\end{equation}
for any $b,c \in A$. Setting in the latter equality $c=1$, we get 
\begin{equation}\label{eq-1a}
\alpha(1,b) = b\alpha(1,1)
\end{equation}
for any $b \in A$. Substituting this back in (\ref{eq-three}), we infer that the
left-hand side of that equality vanish, and hence (\ref{eq-yoyo3}) holds. Due to
(\ref{eq-yoyo4}), $\beta = \Lambda$.

Substituting $b=1$ in (\ref{eq-ffh}), and taking into account (\ref{eq-yoyo4}) 
and (\ref{eq-1a}), we get
$$
\alpha(a,c) = ac\alpha(1,1) + \xi(c)D(a)
$$
for any $a,c \in A$.

Now each of (\ref{eq-fff}), (\ref{eq-ffh}), and (\ref{eq-one}) is equivalent to 
(\ref{eq-yoyo5}), and (\ref{eq-two}) is satisfied automatically.
\end{proof}

\begin{corollary}\label{cor-minus-2}
$\Homol^2_{2} (\ess \otimes A, \ess \otimes A) \simeq A$. The basic cocycles can
be chosen as
\begin{equation}\label{eq-cocycle}
(e \otimes a) \wedge (h \otimes b) \mapsto f \otimes abv
\end{equation}
for some $v\in A$.
\end {corollary}

\begin{proof}
Apply Lemma \ref{lemma-minus-2-ext} with $\xi = 0$, $D=0$, and $\Lambda = 0$. 
It is straightforward to check that the cocycles of the form (\ref{eq-cocycle}) 
are cohomologically independent for linearly independent values of $v$. (This 
also follows from Lemma \ref{lemma-ess}(ii) and observation that the resulting 
map is a $2$-cocycle of kind (\ref{eq-coc-1}), where $\varphi$ coincides with 
(\ref{eq-coc-s-2}).)
\end{proof}

\begin{lemma}\label{lemma-0-ext}
$\Homol^2_0 (\ess \otimes A, \ess \otimes A) \simeq \Har^2(A,A)$.
The basic cocycles can be chosen as
\begin{equation}\label{eq-yoyo}
(x \otimes a ) \wedge (y \otimes b) \mapsto [x,y] \otimes \beta(a,b) ,
\end{equation}
where $x,y\in \ess$, $a,b \in A$, and $\beta$ is a Harrison $2$-cocycle on $A$.
\end{lemma}

\begin{proof}
Let $\Phi$ be a 2-cocycle on $\ess\otimes A$ respecting the standard grading.
We may write
\begin{gather*}
\Phi(h \otimes a, h \otimes b) = h \otimes \alpha(a,b)       \\
\Phi(e \otimes a, h \otimes b) = e \otimes \beta(a,b)        \\
\Phi(f \otimes a, h \otimes b) = f \otimes \gamma(a,b)       \\
\Phi(e \otimes a, f \otimes b) = h \otimes \delta(a,b)       
\end{gather*}
for any $a,b \in A$, and some bilinear maps 
$\alpha, \beta, \gamma, \delta: A \times A \to A$. Due to the fact that $\Phi$ 
is alternating, $\alpha$ is alternating too.

Evaluating $\dcobound \Phi$ for the following 4 triples: 
\begin{gather*}
h \otimes a, h \otimes b, e \otimes c       \\
e \otimes a, e \otimes b, f \otimes c       \\
f \otimes a, f \otimes b, e \otimes c       \\
h \otimes a, e \otimes b, f \otimes c        
\end{gather*}
we get respectively:
\begin{gather}
\beta(ac,b) + \beta(bc,a) + b\beta(c,a) + a\beta(c,b) + c\alpha(a,b) = 0
\label{eq-1} \\
\beta(b,ac) + \beta(a,bc) + b\delta(a,c) + a\delta(b,c) = 0
\label{eq-3} \\
\gamma(b,ac) + \gamma(a,bc) + b\delta(c,a) + a\delta(c,b) = 0
\label{eq-4} \\
\delta(ab,c) + \delta(b,ac) + \alpha(bc,a) + c\beta(b,a) + b\gamma(c,a) = 0
\label{eq-5}
\end{gather}
for any $a,b,c \in A$.

Modifying the cocycle $\Phi$ by the coboundary $\dcobound\Omega$, where the 
cochain $\Omega: \ess \otimes A \to \ess \otimes A$ is defined as follows:
$\Omega(e\otimes A) = \Omega(f\otimes A) = 0$, and
$\Omega(h\otimes a) = h \otimes \beta(1,a)$, we may assume that
$\Phi(e \otimes 1, h \otimes a) = 0$, i.e., $\beta(1,a) = 0$.

Substituting $c=1$ in (\ref{eq-1}) and in (\ref{eq-3}), we get respectively:
\begin{gather}
\alpha(a,b) = \beta(a,b) + \beta(b,a) \label{eq-beta} \\
\beta(a,b) + \beta(b,a) = b\delta(a,1) + a\delta(b,1) \notag
\end{gather}
and, consequently,
\begin{equation}\label{eq-alpha}
\alpha(a,b) = b\delta(a,1) + a\delta(b,1)
\end{equation}
for any $a,b \in A$.

Similarly, substituting $c=1$ in (\ref{eq-4}), we get:
\begin{equation}\label{eq-gamma}
\gamma(a,b) + \gamma(b,a) = b\delta(1,a) + a\delta(1,b)
\end{equation}
for any $a,b \in A$.

Substituting $b=1$ in (\ref{eq-5}), we get:
\begin{equation}\label{eq-gamma-delta}
\delta(a,c) + \delta(1,ac) + \alpha(c,a) + \gamma(c,a) = 0
\end{equation}
for any $a,c \in A$.

Symmetrizing the last equality and using (\ref{eq-gamma}), we get:
$$
\delta(a,c) + \delta(c,a) + a\delta(1,c) + c\delta(1,a) = 0
$$
for any $a,c \in A$.

Setting here $a=1$ and $c=1$ respectively, we get 
$\delta(a,1) = \delta(1,a) = a\delta(1,1)$, and hence $\delta$ is symmetric. 
Together with (\ref{eq-alpha}) this implies $\alpha = 0$, and, according to 
(\ref{eq-beta}), that $\beta$ is symmetric. Moreover, (\ref{eq-gamma-delta}) now
yields 
$$
\gamma(a,b) = \delta(a,b) + ab\delta(1,1)
$$
for any $a,b \in A$, and hence $\gamma$ is symmetric.

Substituting now $b=1$ in (\ref{eq-3}), we get
$$
\beta(a,c) = \delta(a,c) + ac\delta(1,1)
$$
for any $a,c \in A$, and hence $\beta = \gamma$.

Hence, $\Phi$ may be written in the form (\ref{eq-yoyo}), plus the map sending 
$(e \otimes a) \wedge (f \otimes b)$ to $h \otimes ab\delta(1,1)$. The latter 
map is obviously equal to the coboundary $\dcobound \Omega$ for 
$\Omega (e \otimes a) = e \otimes a\delta(1,1)$ and 
$\Omega(f \otimes A) = \Omega(h \otimes A) = 0$.
\end{proof}

\noindent\emph{Conclusion of the proof of Proposition \ref{prop}}.
According to Lemma \ref{lemma-ess}(ii), 
$$
\Homol^2(\ess,\ess) = \Homol^2_{-2}(\ess,\ess) \oplus \Homol^2_2(\ess,\ess) ,
$$
with basic cocycles (\ref{eq-coc-s}) and (\ref{eq-coc-s-2}) of weight
$-2$ and $2$, respectively. This, coupled with Corollary \ref{cor-minus-2} 
(together with its negative grading counterpart), implies that 
$$
\Homol^2_{-2}(\ess \otimes A, \ess \otimes A) \oplus 
\Homol^2_2(\ess \otimes A, \ess \otimes A) \simeq 
\Homol^2(\ess,\ess) \otimes A .
$$
The rest follows from Lemma \ref{lemma-0-ext}.
\end {proof}

\section{
Second cohomology and deformations of $\ess \otimes A$ extended by derivations
}\label{sec-plus-d}

Now, having the low-degree cohomology of the current Lie algebra 
$\ess \otimes A$ at hand, we glue its components together, via the standard 
Hochschild--Serre spectral sequence, to compute the second positive cohomology 
of an extension of $\ess \otimes A$ of the form (\ref{eq-semisimple}), where 
both $\mathfrak D$ and $\mathfrak E$ are $1$-dimensional (in fact, $\mathfrak D$
is linearly spanned by $(\ad f)^2$), i.e., of Lie algebras of the form 
$$
\ess \otimes A + g \otimes U + KD ,
$$
where $D$ is a nonzero derivation of $A$, and $U$ is a $D$-invariant subspace of
$A$. According to Remark \ref{rem-lemma-ess}, part 3, those are exactly Lie 
algebras lying between $\ess \otimes A + KD$ and $W_1(2) \otimes A + KD$.

Following \S \ref{subseq-grad}, we consider the grading of such algebras of 
length $2$ (length $1$ in the case $U = 0$):
\begin{align}\label{eq-grad}
\begin{split}
\text{weight }-1 &: \quad e \otimes A \phantom{.}                  \\
\text{weight }\phantom{-1}0  &: \quad h \otimes A + KD \phantom{.} \\
\text{weight }\phantom{-1}1  &: \quad f \otimes A \phantom{.}      \\
\text{weight }\phantom{-1}2  &: \quad g \otimes U .
\end{split}
\end{align}

It is possible to compute the whole second cohomology, as well as to consider 
cases of more general algebras $\mathfrak D$ and $\mathfrak E$, but this will 
lead to numerous trivial but very cumbersome technicalities. Also, the resulting
formulas are not very aesthetically appealing. That is why we restrict ourselves
just to the case necessary for our present goal. The computations are similar 
to those performed in \cite[\S 5]{me-deformations}.

\begin{proposition}\label{prop-h2-d}
Assume $\dim A > 2$ and $U \ne 0$. Then
\begin{equation*}
\Homol^2_+(\ess \otimes A + g \otimes U + KD, \> \ess \otimes A + g \otimes U + KD)
\simeq A^D \oplus \frac{A}{D(A) + U} \oplus \Xi_{D,U} ,
\end{equation*}
where $\Xi_{D,U}$ consists of linear maps $\xi: A \to K$ vanishing on 
$A^{[2]}$, $D(A)$, and $U$, and satisfying the conditions
$$
\xi(a)D(b) + \xi(b)D(a) \in U
$$
and
\begin{multline}\label{eq-long}
\Big(\xi(a)b + \xi(b)a + \xi(ab)1\Big)D(c) +
\Big(\xi(c)a + \xi(a)c + \xi(ca)1\Big)D(b) +
\Big(\xi(b)c + \xi(c)b + \xi(bc)1\Big)D(a) \\ = 0
\end{multline}
for any $a,b,c \in A$.

The basic cocycles can be chosen as:
\begin{enumerate}[\upshape(i)]

\item \label{item-i} 
$$
(e \otimes a) \wedge (h \otimes b) \mapsto f \otimes abv
$$
for some $v\in A^D$;

\item \label{item-iii}
$$
(e \otimes a) \wedge D \mapsto f \otimes av
$$ 
for some $v \in A$;

\item \label{item-iv}
\begin{align*}
(e \otimes a) \wedge (e \otimes b) &\mapsto 
h \otimes \Big(\xi(a)D(b) + \xi(b)D(a)\Big) + \xi(ab)D
\\
(e \otimes a) \wedge (h \otimes b) &\mapsto f \otimes \xi(b)D(a)
\\
(e \otimes a) \wedge (f \otimes b) &\mapsto 
g \otimes \Big(\xi(a)D(b) + \xi(b)D(a)\Big) 
\\
(h \otimes a) \wedge (h \otimes b) &\mapsto 
g \otimes \Big(\xi(a)D(b) + \xi(b)D(a)\Big) ,
\end{align*}
where $\xi \in \Xi_{D,U}$.

\end{enumerate}
\end{proposition}

The case $U=0$ is similar, if not slightly simpler.

\begin{proposition}\label{prop-u0}
\begin{equation*}
\Homol^2_+(\ess \otimes A + KD, \> \ess \otimes A + KD)
\simeq A^D \oplus \frac{A}{D(A)} \oplus \Xi_{D,0} .
\end{equation*}

The basic cocycles can be chosen as follows: types 
{\rm (\ref{item-i})} and {\rm (\ref{item-iii})} as in 
Proposition \ref{prop-h2-d}, and
\begin{enumerate}
\item[{\rm (iii)}]
\begin{align*}
(e \otimes a) \wedge (e \otimes b) &\mapsto \xi(ab)D             \\
(e \otimes a) \wedge (h \otimes b) &\mapsto f \otimes \xi(b)D(a) ,
\end{align*}
where $\xi \in \Xi_{D,0}$.
\end{enumerate}
\end{proposition} 

\begin{remark}
By definition, $\Xi_{D,0}$ consists of linear maps $\xi: A \to K$ vanishing on 
$A^{[2]}$ and $D(A)$, and such that $\xi(a)D(b) + \xi(b)D(a) = 0$ for any 
$a,b \in A$. 
\end{remark}

\begin{proof}[Proof of Proposition \ref{prop-h2-d}]
By \cite[Proposition 1.1(i)]{me-deformations}, 
$\Z(\ess \otimes A) \simeq \Z(\ess) \otimes A = 0$ (where $\Z(L)$ denotes the 
center of a Lie algebra $L$), so $\ess \otimes A$ is centerless and, obviously, 
perfect. Further, $g \otimes U + KD$ consists, obviously, of outer derivations 
of the algebra $\ess \otimes A$ 
(cf. \cite[Proposition 1.1(ii)]{me-deformations}), so 
\cite[Lemma 5.1]{me-deformations} is applicable. By this Lemma, the relevant
terms in the Hochschild--Serre spectral sequence abutting to the second
cohomology in question, relative to the ideal $\ess \otimes A$, are:
\begin{align}
\E^{20}_\infty &= 0           \notag
\\
\E^{11}_\infty = \E^{11}_2 &\simeq \Homol^1\Big(g \otimes U + KD, \> 
\frac{\Homol^1(\ess \otimes A, \ess \otimes A)}{g \otimes U + KD}\Big)  
\label{eq-e11} \\
\E^{21}_2 &\simeq \Homol^2\Big(g \otimes U + KD, \>
\frac{\Homol^1(\ess \otimes A, \ess \otimes A)}{g \otimes U + KD}\Big)  
\label{eq-e21} \\
\E^{02}_2 &\simeq \Homol^2(\ess \otimes A, \ess \otimes A)^{(g \otimes U + KD)} 
\oplus (\Ker F)^{(g \otimes U + KD)} \>
\notag \\
\E^{02}_\infty = \E^{02}_3 &= \Ker \dcobound^{02}_2 . \notag
\end{align}
(The formula for the $\E^{21}_2$ term is not specified in 
\cite[Lemma 5.1]{me-deformations}, but it is implicit in its proof, and follows 
from the same standard homological considerations as for the $\E^{11}_2$ term.)
The linear map 
$$
F: \Homol^2(\ess \otimes A,K) \otimes (U + KD) \to 
\Homol^3(\ess \otimes A, \ess \otimes A)
$$
sends the elements $[\Phi] \otimes u$ and $[\Phi] \otimes D$, where $[\Phi]$ is 
the class of a $2$-cocycle $\Phi$, and $u\in U$, to the class of a $3$-cocycle 
defined, respectively, as
\begin{equation*}
(x \otimes a) \wedge (y \otimes b) \wedge (z \otimes c) \mapsto 
\Phi(x \otimes a,y \otimes b) [z,g] \otimes cu + 
\Phi(z \otimes c,x \otimes a) [y,g] \otimes bu + 
\Phi(y \otimes b,z \otimes c) [x,g] \otimes au
\end{equation*}
and
\begin{equation*}
(x \otimes a) \wedge (y \otimes b) \wedge (z \otimes c) \mapsto 
\Phi(x \otimes a,y \otimes b) z \otimes D(c) + 
\Phi(z \otimes c,x \otimes a) y \otimes D(b) + 
\Phi(y \otimes b,z \otimes c) x \otimes D(a)
\end{equation*}
for any $x,y,z\in \ess$, $a,b,c\in A$ (this map arises as a connection 
homomorphism in the appropriate cohomological long exact sequence, see proof of 
\cite[Lemma 5.1]{me-deformations} for details). 

Assuming that the element 
\begin{equation}\label{eq-elem}
\sum_i [\Phi_i] \otimes u_i + [\Phi_D] \otimes D
\end{equation}
lies in $\Ker F$, the corresponding element of $\E^{02}_2$ is generated by a
cochain of the form $\sum_i \Phi_i^\prime + \Phi_D^\prime + \widetilde{\Phi}$, 
where $\Phi_i^\prime, \Phi_D^\prime$ are $(g \otimes U + KD)$-valued 
$2$-cochains built from $\Phi_i, \Phi_D$:
\begin{align*}
\Phi_i^\prime (x \otimes a, y \otimes b) &= 
\Phi_i(x \otimes a, y \otimes b) g \otimes u_i \\
\Phi_D^\prime (x \otimes a, y \otimes b) &= 
\Phi_D(x \otimes a, y \otimes b) D ,
\end{align*}
and $\widetilde{\Phi}$ is a $2$-cochain on $\ess \otimes A$ such that 
$F\Phi = \dcobound \widetilde{\Phi}$. Note that when choosing the basic cocycles
induced by this part, we should exclude those corresponding to the case
$F\Phi = 0$, as such cocycles are already accounted for by the 
$\Homol^2(\ess \otimes A, \ess \otimes A)^{(g \otimes U + KD)}$ summand of 
$\E^{02}_2$.

The action of $D$ and of $g \otimes U$ on these $2$-cochains is determined,
respectively, by:
\begin{align}\label{eq-gu-act}
\begin{split}
(D\widetilde{\Phi})(x \otimes a, y \otimes b) &= 
\widetilde{\Phi}(x \otimes D(a), y \otimes b) + 
\widetilde{\Phi}(x \otimes a, y \otimes D(b)) +
D(\widetilde{\Phi}(x \otimes a, y \otimes b))
\\
(D\Phi_i^\prime)(x \otimes a, y \otimes b) &\\
= g& \otimes \Big(
\Big(\Phi_i(x \otimes D(a), y \otimes b) + \Phi_i(x \otimes a, y \otimes D(b))\Big) u_i 
+ \Phi_i(x \otimes a, y \otimes b) D(u_i) \Big)
\\
(D\Phi_D^\prime)(x \otimes a, y \otimes b) &= 
\Big(\Phi_D(x \otimes D(a), y \otimes b) + \Phi_D(x \otimes a, y \otimes D(b))\Big)D
\end{split}
\end{align}
and
\begin{align}\label{eq-d-act}
\begin{split}
((g \otimes u)\widetilde{\Phi})(x \otimes a, y \otimes b) &= 
\widetilde{\Phi}([x,g] \otimes au, y \otimes b) + 
\widetilde{\Phi}(x \otimes a, [y,g] \otimes bu) +
[\widetilde{\Phi}(x \otimes a, y \otimes b), g \otimes u]
\\
((g \otimes u)\Phi_i^\prime)(x \otimes a, y \otimes b) &= g \otimes
\Big(\Phi_i([x,g] \otimes au, y \otimes b) + \Phi_i(x \otimes a, [y,g] \otimes bu)\Big)
u_i
\\
((g \otimes u)\Phi_D^\prime)(x \otimes a, y \otimes b) &= 
\Phi_D (x \otimes a, y \otimes b) g \otimes D(u) \\&+
\Big(\Phi_D([x,g] \otimes au, y \otimes b) + \Phi_D(x \otimes a, [y,g] \otimes bu)\Big) D
\end{split}
\end{align}
for any $x,y\in \ess$, $a,b\in A$.

The grading (\ref{eq-grad}) of the underlying Lie algebra induces a 
$\mathbb Z$-grading on each term of the spectral sequence (see, for example, 
\cite[p.~44 of the English edition]{fuchs}), and we will concentrate on the 
positive part which abuts to the second positive cohomology in question.

By Proposition \ref{prop} (for $n=1$), the module in 
(\ref{eq-e11}) and (\ref{eq-e21}) is isomorphic to 
\begin{equation}\label{eq-mod}
\frac{\Homol^1(\ess,\ess) \otimes A}{g \otimes U} \oplus \frac{\Der(A)}{KD} ,
\end{equation}
and $g \otimes U$ acts on it trivially. 

The algebra $g \otimes U + KD$ is 
$\mathbb Z$-graded with components of weight $0$ and $2$, and its module 
(\ref{eq-mod}) is $\mathbb Z$-graded with components $-2$, $0$ and $2$. Hence, 
in degree $1$, the corresponding cochain complex has components of weight 
$-4$, $-2$, $0$, $2$, and in degree $2$ it has components of weight 
$-6$, $-4$, $-2$, $0$, $2$. 

In degree $2$, the positive part of the second degree cohomology in 
(\ref{eq-e21}) is generated by the cochains mapping a pair of elements of weight
$0$ to elements of weight $2$, but since the zero component, $KD$, is 
$1$-dimensional, this positive cohomology vanishes. Hence, the positive part
of our Hochschild--Serre spectral sequence stabilizes at the $\E_2$ term, and
$$
\Homol^2_+(\ess \otimes A + g \otimes U + KD, \> 
           \ess \otimes A + g \otimes U + KD) 
\simeq (\E^{11}_2)_+ \oplus (\E^{02}_2)_+ .
$$

In degree $1$, elementary reasonings based on the Hochschild--Serre spectral 
sequence abutting to the first degree cohomology in (\ref{eq-e11}) with respect
to ideal $g \otimes U$, yield
\begin{equation*}
\E^{11}_2 \simeq 
\frac{\Homol^1(\ess,\ess) \otimes A}{\Homol^1(\ess,\ess) \otimes D(A) + g \otimes U}
\oplus \frac{\Der(A)}{[D,\Der(A)] + KD} \oplus
\Hom\Big(g \otimes U, 
\frac{\Homol^1(\ess,\ess) \otimes A}{g \otimes U} \oplus \frac{\Der(A)}{KD}\Big)^D .
\end{equation*}
The positive component here, of weight $2$, is
$$
(\E^{11}_2)_+ \simeq 
\Homol^1_+(\ess,\ess) \otimes \frac{A}{D(A) + U} \simeq \frac{A}{D(A) + U} ,
$$
with the corresponding basic cocycles of type (\ref{item-iii}).

Using Proposition \ref{prop} (for $n=2$), we infer that the first summand in the
$\E^{02}_2$ term is isomorphic to 
\begin{equation}\label{eq-summand1}
\Big(\Big(\Homol^2(\ess,\ess) \otimes A^D\Big) \oplus \Har^2(A,A)^D\Big)^{g \otimes U} .
\end{equation}
The corresponding $2$-cocycles on $\ess \otimes A + g \otimes U + D$ are 
constructed as the maps $(\ess \otimes A) \wedge (\ess \otimes A) \to \ess \otimes A$, and the 
positive part of (\ref{eq-summand1}) is 
\begin{equation}\label{eq-yoyo2}
\Big(\Homol^2_+(\ess,\ess) \otimes A^D\Big)^{g \otimes U} \simeq
\Homol^2_+(\ess,\ess)^g \otimes A^D +
\Homol^2_+(\ess,\ess) \otimes (A^D \cap \set{a\in A}{aU = 0}) .
\end{equation}
It is readily verified that the only basic cocycle (\ref{eq-coc-s-2}) in 
$\Homol^2_+(\ess,\ess)$ is $g$-invariant, so (\ref{eq-yoyo2}) is isomorphic to 
$A^D$. The corresponding basic cocycles are of type (\ref{item-i}).

The $(\Ker F)^{(g \otimes U + KD)}$ part of $\E^{02}_2$, being a subspace of the
tensor product of $\mathbb Z$-graded spaces \linebreak $\Homol^2(\ess \otimes A,K)$ (with 
weights $-2$, $0$, $2$, see \S \ref{sec-h2-zass}), and $g \otimes U + KD$ (with 
weights $0$ and $2$), is $\mathbb Z$-graded itself, with graded components of 
weight $-2$, $0$, $2$, $4$, which can be evaluated separately. As we are 
interested in the positive part, we deal with weights $2$ and $4$ only.

A general element $\Phi$ of $\Ker F$ of the form (\ref{eq-elem}) (where $u_i$'s 
may be assumed to be linearly independent) is of weight $2$, if the cocycles 
$\Phi_i$ are of weight $0$, and $\Phi_D$ is of weight $2$. According to 
Lemma \ref{lemma-h2-triv-0} and Corollary \ref{cor-h2-triv-+}, we have
\begin{align*}
\Phi_i(e \otimes a, f \otimes b) &= \alpha_i(a,b)  \\
\Phi_i(h \otimes a, h \otimes b) &= \alpha_i(a,b)  \\
\Phi_D(e \otimes a, e \otimes b) &= \xi(ab)
\end{align*}
for any $a,b \in A$, and some $\alpha_i \in \widehat{\HC^1}(A)$ and a linear 
map $\xi: A \to K$ vanishing on $A^{[2]}$. Then $F\Phi$ has the following form:
\begin{align*}
F\Phi(e\otimes a, e\otimes b, e\otimes c) &= e \otimes 
       \Big(\xi(ab)D(c) + \xi(ca)D(b) + \xi(bc)D(a)\Big)            \\
F\Phi(e\otimes a, e\otimes b, h\otimes c) &= h \otimes \xi(ab)D(c)  \\
F\Phi(e\otimes a, e\otimes b, f\otimes c) &= f \otimes \Big(
       \xi(ab)D(c) + \sum_i \Big(a\alpha_i(b,c) + b\alpha_i(a,c) \Big)u_i \Big) 
\\
F\Phi(e\otimes a, h\otimes b, h\otimes c) &= f \otimes a\sum_i \alpha_i(b,c)u_i .
\end{align*}
for any $a,b,c \in A$.

Denoting $\Lambda(a,b) = \sum_i \alpha_i(a,b) u_i$, we are in situation of 
Lemma \ref{lemma-minus-2-ext}. That Lemma tells the necessary and sufficient 
conditions for $F\Phi$ to vanish in $\Homol^3(\ess \otimes A,\ess \otimes A)$, 
i.e., to be a $3$-coboundary (of weight $2$). In particular, 
\begin{equation}\label{eq-lambda}
\Lambda(a,b) = \xi(a)D(b) + \xi(b)D(a)
\end{equation}
belongs to $U$ for any $a,b\in A$.
From the same Lemma it follows that $\widetilde{\Phi}$ has the form 
\begin{align*}
\widetilde{\Phi}(e \otimes a, e \otimes b) &= h \otimes 
\Big(\xi(a)D(b) + \xi(b)D(a)\Big)
\\
\widetilde{\Phi}(e \otimes a, h \otimes b) &= 
f \otimes \Big(abv + \xi(b)D(a)\Big)
\end{align*}
for some $v\in A$. 

The condition of invariance of 
$\sum_i \Phi_i^\prime + \Phi_D^\prime + \widetilde{\Phi}$ with respect to the 
$D$- and $g \otimes U$-actions (\ref{eq-d-act}) and (\ref{eq-gu-act}) yields, 
respectively, $\xi(D(A)) = 0$ and $D(v) = 0$, and
$$
\Big(\xi(b)u + \xi(bu)1\Big)D(a) + \Big(\xi(a)u + \xi(au)1\Big)D(b) + 
\Big(\xi(a)b + \xi(b)a + \xi(ab)1\Big)D(u) = 0
$$
for any $a,b\in A$, $u \in U$. Coupled with (\ref{eq-yoyo5}), and the fact that
$D$ is nonzero, the latter equality is equivalent to $\xi(U) = 0$. Excluding the
cocycles corresponding to the case $F\Phi = 0$, i.e., where $\xi = 0$ and $D=0$
(those are exactly cocycles of type (\ref{item-i})), we arrive at basic cocycles
of type (\ref{item-iv}).

A general element $\Phi$ of $\Ker F$ of the form (\ref{eq-elem}) is of 
weight $4$, if $\Phi_i$ are of weight $2$, and $\Phi_D = 0$. According to 
Corollary \ref{cor-h2-triv-+}, we have:
$$
\Phi_i(e \otimes a, e \otimes b) = \xi_i(ab)
$$
for any $a,b \in A$, where $\xi_i: A \to K$ are linear maps vanishing on 
$A^{[2]}$. Then $F\Phi$ has the following form:
$$
F\Phi(e \otimes a, e \otimes b, e \otimes c) = f \otimes 
\sum_i \Big(\xi_i(ab)c + \xi_i(ca)b + \xi_i(bc)a\Big)u_i .
$$
If $F\Phi = \dcobound \widetilde{\Phi}$ for some $2$-cochain 
$\widetilde{\Phi} \in \C^2(\ess \otimes A, \ess \otimes A)$, then the latter 
cochain should also have weight $4$, and hence vanishes. Thus $F\Phi = 0$, and 
the linear independence of $u_i$'s implies
$$
\xi_i(ab)c + \xi_i(ca)b + \xi_i(bc)a = 0
$$
for any $i$ and $a,b,c\in A$. Taking in the last equality $a,b,c$ to be linearly
independent, we get $\xi_i(ab) = 0$ for any two linearly independent $a,b\in A$.
Coupled with the condition $\xi_i(A^{[2]}) = 0$, this implies $\xi_i = 0$.
\end{proof}

\begin{proof}[Proof of Proposition \ref{prop-u0}]
All the reasonings above are valid verbatim (and are somewhat simplified) in 
the case $U=0$, except the part concerned with evaluation of $(\Ker F)^D$ which
leads to cocycles of type (\ref{item-iv}) as in Proposition \ref{prop-h2-d}. In that 
case $\Lambda = 0$, so (\ref{eq-lambda}) implies $\xi(a)D(b) + \xi(b)D(a) = 0$, from
which the condition (\ref{eq-long}) follows, and we arrive at the cocycles of
type (\ref{item-iv}) as in Proposition \ref{prop-u0}.
\end{proof}

\begin{proposition}\label{prop-deform}
If $\dim A > 2$ and $U \ne 0$, any filtered deformation of the algebra 
$\ess \otimes A + g \otimes U + KD$ is isomorphic to a Lie algebra with the 
following bracket (assuming $a,b$ are arbitrary elements of $A$, and $u,t$ are
arbitrary elements of $U$):
\begin{alignat*}{2}
&\{e \otimes a, e \otimes b \} \>\>&=&\>\> 
h \otimes \Big(\xi(a)D(b) + \xi(b)D(a)\Big) + g \otimes \lambda(ab) + \xi(ab)D
\\
&\{e \otimes a, h \otimes b \} \>\>&=&\>\>
e \otimes ab + f \otimes \Big(abv + \xi(b)D(a)\Big)
\\
&\{e \otimes a, f \otimes b \} \>\>&=&\>\>
h \otimes ab + g \otimes \Big(\xi(a)D(b) + \xi(b)D(a)\Big)
\\
&\{e \otimes a, g \otimes u \} \>\>&=&\>\> f \otimes au
\\
&\{h \otimes a, h \otimes b \} \>\>&=&\>\> 
g \otimes \Big(\xi(a)D(b) + \xi(b)D(a)\Big)          
\\
&\{h \otimes a, f \otimes b \} \>\>&=&\>\> f \otimes ab                 
\\
&\{h \otimes a, g \otimes u \} \>\>&=&\>\> 0
\\
&\{f \otimes a, f \otimes b \} \>\>&=&\>\> 0
\\
&\{f \otimes a, g \otimes u \} \>\>&=&\>\> 0
\\
&\{g \otimes u, g \otimes t \} \>\>&=&\>\> 0
\\
&\{e \otimes a, D \} \>\>&=&\>\> e \otimes D(a) + f \otimes aw          
\\
&\{h \otimes a, D \} \>\>&=&\>\> h \otimes D(a)                         
\\
&\{f \otimes a, D \} \>\>&=&\>\> f \otimes D(a) 
\\
&\{g \otimes u, D \} \>\>&=&\>\> g \otimes D(u) 
\end{alignat*}
for some $v,w \in A$ such that $D(v) = 0$, $\xi \in \Xi_{D,U}$, and a linear map
$\lambda: A \to U$ such that $\lambda(A^{[2]}) = 0$, and
\begin{multline}\label{eq-lambda1}
\Big(\xi(ab)D(c) + \xi(ca)D(b) + \xi(bc)D(a)\Big)v + 
\Big(\xi(ab)c + \xi(ca)b + \xi(bc)a\Big)w                   \\
= c\lambda(ab) + b\lambda(ca) + a\lambda(bc) 
\end{multline}
\begin{equation}   \label{eq-lambda2} 
\xi(a)D(w) + \xi(w)D(a) = \lambda(D(a)) + D(\lambda(a)) .
\end{equation}
for any $a,b,c\in A$. 
\end{proposition}

\begin{proof}
By Proposition \ref{prop-h2-d}, any such infinitesimal filtered deformation is a
linear combination of the cocycles of type (\ref{item-i}) for some $v \in A^D$, 
type (\ref{item-iii}) for some $w\in A$, and type (\ref{item-iv}) for some 
$\xi \in \Xi_{D,U}$, denoted by $\Phi_v$, $\Psi_w$, and $\Upsilon_\xi$, 
respectively. As all these cocycles are of weight $2$, all the possible Massey 
brackets between them are of weight $4$. Thus nonzero values of these brackets are possible only
when either all the $3$ arguments are of weight $-1$, or two arguments are of
weight $-1$, and one of weight $0$, i.e., for triples of the form
\begin{align*}
&e \otimes a, \> e \otimes b, \> e \otimes c  \\
&e \otimes a, \> e \otimes b, \> h \otimes c  \\
&e \otimes a, \> e \otimes b, \> D .
\end{align*}

Direct computation shows that the only possibly nonvanishing Massey brackets
on these triples are:
\begin{align*}
[[\Phi_v, \Upsilon_\xi]] (e \otimes a, e \otimes b, e \otimes c) &=
f \otimes \Big(\xi(ab)D(c) + \xi(ca)D(b) + \xi(bc)D(a)\Big)v
\\
[[\Psi_w, \Upsilon_\xi]] (e \otimes a, e \otimes b, e \otimes c) &=
f \otimes \Big(\xi(ab)c + \xi(ca)b + \xi(bc)a\Big)w
\\
[[\Upsilon_\xi,\Psi_w]] (e \otimes a, e \otimes b, D) &=
g \otimes \Big(\xi(ab)D(w) + \xi(w)D(ab)\Big) .
\end{align*}

Each $2$-cochain $\Omega$ on the algebra in question of weight $4$ has the form
$$
\Omega(e \otimes a, e \otimes b) = g \otimes \alpha(a,b)
$$
for some alternating bilinear map $\alpha: A \times A \to U$, and the 
corresponding coboundary $\dcobound\Omega$ has the form
\begin{align*}
\dcobound\Omega(e \otimes a, e \otimes b, e \otimes c) &= 
f \otimes \Big(a\alpha(b,c) + c\alpha(a,b) + b\alpha(c,a)\Big)
\\
\dcobound\Omega(e \otimes a, e \otimes b, h \otimes c) &= 
g \otimes \Big(\alpha(ac,b) + \alpha(bc,a)\Big)
\\
\dcobound\Omega(e \otimes a, e \otimes b, D) &= 
g \otimes \Big(\alpha(D(a),b) + \alpha(a,D(b)) + D(\alpha(a,b))\Big) .
\end{align*}

Hence, the necessary and sufficient condition for the second order 
prolongability of the infinitesimal deformation $\Phi_v + \Psi_w + \Upsilon_\xi$, is the 
existence of $\alpha$ such that
\begin{alignat}{2}
&\Big(\xi(ab)D(c) + \xi(ca)D(b) &+\>& \xi(bc)D(a)\Big)v +
\Big(\xi(ab)c + \xi(ca)b + \xi(bc)a\Big)w
\notag \\
&                            &\>=\>& a\alpha(b,c) + c\alpha(a,b) + b\alpha(c,a)
\notag \\
&\alpha(ac,b) + \alpha(bc,a) &\>=\>& 0  \label{eq-alpha1}
\\
&\xi(ab)D(w) + \xi(w)D(ab) &\>=\>& \alpha(D(a),b) + \alpha(a,D(b)) + D(\alpha(a,b))
\notag
\end{alignat}
for any $a,b,c \in A$. Among these $3$ conditions, (\ref{eq-alpha1}) is 
equivalent to $\alpha(a,b) = \lambda(ab)$ for some linear map 
$\lambda: A \to U$, and the other two conditions are equivalent then to 
(\ref{eq-lambda1}) and (\ref{eq-lambda2}), respectively.

The Massey brackets between $2$-cochains $\Phi_v$, $\Psi_w$, $\Upsilon_\xi$ of 
weight $2$, and $2$-cochain $\alpha$ of weight $4$, as well as between
$\alpha$ and $\alpha$, have weight $6$ and $8$, respectively, and hence vanish. 
Consequently, the second order deformation, if exists, is prolonged trivially to
a global one.
\end{proof}

The case $U=0$ is easier:

\begin{proposition}\label{prop-deform-u0}
Any filtered deformation of the algebra $\ess \otimes A + KD$ is isomorphic to 
the Lie algebra with the following bracket (assuming $a,b\in A$):
\begin{alignat*}{2}
&\{e \otimes a, e \otimes b \} \>\>&=&\>\> \xi(ab)D
\\
&\{e \otimes a, h \otimes b \} \>\>&=&\>\>
e \otimes ab + f \otimes \Big(abv + \xi(b)D(a)\Big)
\\
&\{e \otimes a, f \otimes b \} \>\>&=&\>\> h \otimes ab
\\
&\{h \otimes a, h \otimes b \} \>\>&=&\>\> 0
\\
&\{h \otimes a, f \otimes b \} \>\>&=&\>\> f \otimes ab                 
\\
&\{f \otimes a, f \otimes b \} \>\>&=&\>\> 0
\\
&\{e \otimes a, D \} \>\>&=&\>\> e \otimes D(a) + f \otimes aw          
\\
&\{h \otimes a, D \} \>\>&=&\>\> h \otimes D(a)                         
\\
&\{f \otimes a, D \} \>\>&=&\>\> f \otimes D(a) 
\end{alignat*}
for some $v,w \in A$ such that $D(v) = 0$, and $\xi \in \Xi_D$ such that
\begin{equation}\label{eq-xi}
\Big(\xi(ab)c + \xi(ca)b + \xi(bc)a\Big)w = 0
\end{equation}
for any $a,b,c\in A$.
\end{proposition}

\begin{proof}
Similar to Proposition \ref{prop-deform}. Since, in this case, there are no 
$2$-cochains of weight $4$, the necessary and sufficient condition for 
prolongability of infinitesimal deformation is
\begin{equation*}
\Big(\xi(ab)D(c) + \xi(ca)D(b) + \xi(bc)D(a)\Big)v +
\Big(\xi(ab)c + \xi(ca)b + \xi(bc)a\Big)w = 0
\end{equation*}
for any $a,b,c\in A$. The first summand here vanishes, and we are left with 
(\ref{eq-xi}).
\end{proof}

To determine the isomorphism classes of algebras appearing in Propositions 
\ref{prop-deform} and \ref{prop-deform-u0}, yet alone to determine which of 
the corresponding $2$-cocycles are semitrivial, in the parlance of \cite{blls}
(see Remark \ref{rem-lemma-ess}, part 4), seems to be a technically difficult
task. A very partial solution of this task -- with inconclusive answer, but 
enough to achieve our main goal -- is presented in the next section.

\section{
Simple Lie algebras of absolute toral rank $2$ having a Cartan subalgebra of
toral rank $1$}\label{sec-fin}

We arrive, finally, at our main goal:

\begin{theorem}\label{th-main}
Any finite-dimensional simple Lie algebra over an algebraically closed field, of
absolute toral rank $2$, and having a Cartan subalgebra of toral rank $1$, is 
isomorphic to $\ess$.
\end{theorem}

\begin{proof}
Let $\mathscr L$ be a Lie algebra satisfying the conditions of the theorem. 
Apply Theorem \ref{th-skr}. If $\mathscr L$ is isomorphic to Zassenhaus or 
Hamiltonian algebra, the absolute toral rank values of these algebras exclude 
all the cases except $\mathscr L \simeq \ess$. 

Suppose $\mathscr L$ is a filtered deformation of a Lie algebra $L$ satisfying 
(\ref{eq-form}). Then we have:
$$
\TR(S) \le \TR(S \otimes \mathcal O_1(n)) \le \TR(L) \le \TR(\mathscr L) = 2 .
$$
The first two inequalities here follow from the obvious fact that the absolute 
toral rank of an algebra is not less than the absolute toral rank of its 
subalgebra, and the third one follows from \cite[Theorem 5.1]{skryabin} (see 
also \cite[Vol. I, Theorem 1.4.6]{strade-intro}). Therefore, $\TR(S) = 1$ or 
$2$. The first possibility is ruled out by \cite[Theorem 6.5]{skryabin} (an 
alternative, and shorter, proof of nonexistence of simple Lie algebras of absolute toral rank $1$ is given in 
\cite[Theorem 2]{G}), and the second one implies, again, $S \simeq \ess$. 
Consequently, $L$ is a graded semisimple Lie algebra lying between 
$\ess \otimes \mathcal O_1(2)$ and 
$\Der(\ess) \otimes \mathcal O_1(2) + K\partial$.

According to Lemma \ref{lemma-ess}(\ref{en-ess-1}), the algebra of outer 
derivations of $\ess$ is $2$-dimensional abelian, spanned by 
$(\ad e)^2$ and $(\ad f)^2$. Accordingly, the spaces 
$(\ad e)^2 \otimes A$ and $(\ad f)^2 \otimes A$ of derivations of 
$\ess \otimes A$, consist of derivations of weight $-2$ and $2$, respectively. 
On the other hand, $\partial$, considered as a derivation
of $\ess \otimes A$, has weight $0$. As the grading of $L$ has depth $1$,
this rules out derivations belonging to $(\ad e)^2 \otimes A$, and since the 
induced grading on $\ess \otimes \mathcal O_1(2)$ is standard, the algebra $L$ 
necessary has the form 
$$
L = \ess \otimes \mathcal O_1(2) + (\ad f)^2 \otimes U + K\partial  ,
$$
where $U$ is a $\partial$-invariant subspace of $\mathcal O_1(2)$. In other 
words (see discussion at the beginning of \S \ref{sec-plus-d}),
\begin{equation}\label{eq-incl}
\ess \otimes \mathcal O_1(2) + K\partial \subseteq L \subseteq 
W_1(2) \otimes \mathcal O_1(2) + K\partial .
\end{equation}
In particular, $13 \le \dim L \le 17$.

We are in situation of Propositions \ref{prop-deform} and \ref{prop-deform-u0}
with $A = \mathcal O_1(2)$ and $D = \partial$, and the rest of the proof 
consists of a somewhat boring elucidation of the structure of algebras appearing
there. 

Note that in both cases $U \ne 0$ and $U = 0$, the space $\Xi_{\partial,U}$ must
be nonzero, otherwise by Propositions \ref{prop-deform} and 
\ref{prop-deform-u0}, $\mathscr L$ contains an ideal isomorphic to a filtered 
deformation of $\ess \otimes \mathcal O_1(2) + g \otimes U$ (of $\ess \otimes \mathcal O_1(2)$ 
in the case $U=0$), and hence is not simple. Since, by definition, any element 
$\xi \in \Xi_{\partial,U}$ vanishes on 
$\partial(\mathcal O_1(2)) = \langle 1,x,x^{(2)} \rangle$, we should have, 
up to a scalar, $\xi(x^{(3)}) = 1$. We have then
$x = \xi(x^{(2)})\partial(x^{(3)}) + \xi(x^{(3)})\partial(x^{(2)}) \in U$,
which excludes the case $U=0$. On the other hand, in the case 
$U = \mathcal O_1(2)$, the space $\Xi_{\partial,O_1(2)}$ is equal to
zero, since its elements vanish on $\mathcal O_1(2)$. 

Successive application of $\partial$ to elements of $\mathcal O_1(2)$ shows that
any proper $\partial$-invariant subspace of $\mathcal O_1(2)$ containing $x$
coincides with either $\langle 1, x \rangle$ or $\langle 1, x, x^{(2)} \rangle$.

It is straightforward to check that the map $\xi: \mathcal O_1(2) \to K$ 
defined, as above, by
$$
1 \mapsto 0, \quad x \mapsto 0, \quad x^{(2)} \mapsto 0, \quad 
x^{(3)} \mapsto 1 , 
$$
satisfies (\ref{eq-long}). As $\mathcal O_1(2)^{[2]} = K$, and 
$U \subseteq \partial(\mathcal O_1(2))$, the rest of the defining conditions
of $\Xi_{\partial,U}$ are satisfied automatically, so the latter space is
$1$-dimensional, linearly spanned by the map $\xi$.  

Further, since $\Ker\partial = K1$, for the element $v \in A$ from Proposition 
\ref{prop-deform} we have $v = \alpha 1$ for some $\alpha \in K$. The element
$w \in A$ from the same Proposition originates from cocycles of type
(\ref{item-iii}) in Proposition \ref{prop-h2-d}. Up to coboundaries, we may
assume that $w$ lies in a subspace of $\mathcal O_1(2)$ complementary to 
$\partial(\mathcal O_1(2)) + U = \langle 1,x,x^{(2)} \rangle$, and hence
$w = \beta x^{(3)}$ for some $\beta \in K$.

Denote, for a moment, the Lie algebra appearing in our case of 
Proposition \ref{prop-deform} by $\mathscr L(\alpha,\beta,\xi,\lambda)$. For any
nonzero $\gamma \in K$, the map
\begin{alignat*}{2}
e \otimes a &\mapsto \frac{1}{\sqrt{\gamma}} &e \otimes a  \\
h \otimes a &\mapsto &h \otimes a                          \\
f \otimes a &\mapsto \sqrt{\gamma} &f \otimes a            \\
g \otimes u &\mapsto {\hskip 10pt} \gamma &g \otimes u     \\
D           &\mapsto D &
\end{alignat*}
provides us with an isomorphism 
$$
\mathscr L\Big(\alpha, \beta, \xi, \lambda\Big) \simeq 
\mathscr L\Big(\frac{\alpha}{\gamma}, \frac{\beta}{\gamma}, \frac{1}{\gamma}\xi, \frac{1}{\gamma^2}\lambda\Big) ,
$$
so we indeed can normalize $\xi$ by setting $\xi(x^{(3)}) = 1$.

The equality (\ref{eq-lambda1}) for the triple $1,x,x^{(3)}$ yields 
$x\lambda(x^{(3)}) + x^{(3)}\lambda(x) = \alpha 1$. The left-hand side here
belongs to the ideal $\langle x,x^{(2)},x^{(3)} \rangle$ of $\mathcal O_1(2)$, 
hence both sides vanish, and $\alpha = 0$. The remaining part of 
(\ref{eq-lambda1}) can be rewritten as
\begin{equation}\label{eq-der}
c\Lambda(ab) + b\Lambda(ca) + a\Lambda(bc) = 0
\end{equation}
for any $a,b,c \in A$, where $\Lambda(a) =  \lambda(a) + \beta\xi(a)x^{(3)}$. The condition 
(\ref{eq-der}) is equivalent to $\Lambda$ being a derivation of 
$\mathcal O_1(2)$. Consequently, 
$$
\lambda(a) + \beta\xi(a)x^{(3)} = f\partial(a) + g\partial^2(a)
$$ 
for any $a \in A$, and some elements $f,g\in \mathcal O_1(2)$. The obvious 
computations then yield that this is possible if and only if
\begin{equation}\label{eq-cond1}
x^{(2)} \lambda(x) + x \lambda(x^{(2)}) + \lambda(x^{(3)}) = \beta x^{(3)} .
\end{equation}

Writing the equality (\ref{eq-lambda2}) for $a=x, x^{(2)}, x^{(3)}$, we get
respectively:
\begin{alignat*}{2}
&\partial(\lambda(x)) \>&=&\> \beta 1                   \\
&\partial(\lambda(x^{(2)})) \>&=&\> \lambda(x) + \beta x \\
&\partial(\lambda(x^{(3)})) \>&=&\> \lambda(x^2) .
\end{alignat*}
Coupled with (\ref{eq-cond1}), this gives
\begin{alignat*}{2}
&\lambda(1) \>&=&\> 0 \\
&\lambda(x) \>&=&\>       \gamma 1 + \beta x  \\
&\lambda(x^{(2)}) \>&=&\> \delta 1 + \gamma x  \\
&\lambda(x^{(3)}) \>&=&\> \delta x + \gamma x^{(2)}
\end{alignat*}
for some $\gamma, \delta \in K$.

Now let us look at the pairs in the multiplication table of the deformed algebra
$\mathscr L$ in Proposition \ref{prop-deform}, whose product contains elements 
from $g \otimes U$: 
$\{e \otimes a, e \otimes b\}$, 
$\{e \otimes a, f \otimes b\}$,
$\{h \otimes a, h \otimes b\}$, and 
$\{g \otimes u, \partial\}$. Considered on the basic elements of 
$\mathcal O_1(2)$, the term $\xi(a)\partial(b) + \xi(b)\partial(a)$, occurring 
in $\{e \otimes a, f \otimes b\}$ and $\{h \otimes a, h \otimes b\}$, is nonzero
only when one of $a,b$ is equal to $x^{(3)}$, and the other one belongs to 
$\langle x,x^{(2)} \rangle$, and hence the values of that term lie in 
$\langle 1,x \rangle$. The elements of $\partial(U)$, occurring in 
$\{g \otimes u, \partial\}$, also lie in $\langle 1,x \rangle$. Thus
the only occurrences of elements of the form $g \otimes u$, where 
$u \notin \langle 1,x \rangle$, is possible in the product 
$\{e \otimes a, e \otimes b\}$, and all the relevant values are:
\begin{equation*}
\{e \otimes 1, e \otimes x^{(3)} \} = \{e \otimes x, e \otimes x^{(2)} \} =
g \otimes (\delta x + \gamma x^{(2)}) + \partial .
\end{equation*}

Consequently, the commutant $[\mathscr L, \mathscr L]$ lies in
$$
\ess \otimes \mathcal O_1(2) + g \otimes \langle 1,x \rangle + 
\langle g \otimes (\delta x + \gamma x^{(2)}) + \partial \rangle ,
$$
and hence is of dimension $15$ at most (in fact, as it is easy to see from  
considerations below, it is of dimension $15$). This excludes the case 
$U = \langle 1,x,x^{(2)} \rangle$, where $\dim \mathscr L = 16$.

In the remaining case we have $U = \langle 1,x \rangle$, $\dim \mathscr L = 15$,
and $\gamma = 0$. Let us now write the multiplication table of the deformed 
algebra explicitly. We list only those multiplications between the basic 
elements whose product in the deformed algebra $\mathscr L$ differs from the 
product in the graded algebra $L$.
\begin{alignat}{5}
&&&\{e \otimes 1,\> &e &\otimes x &\} \>&=&\> g &\otimes \beta x
\notag \\
&&&\{e \otimes 1,\> &e &\otimes x^{(2)} &\} \>&=&\> g &\otimes \delta 1
\notag \\
&\{e \otimes 1,\> &e \otimes x^{(3)} \} \>=\>\> 
&\{e \otimes x,\> &e &\otimes x^{(2)} &\} \>&=&\> g &\otimes \delta x + \partial
\notag \\
&&&\{e \otimes x,\> &e &\otimes x^{(3)} &\} \>&=&\> h &\otimes 1
\notag \\
&&&\{e \otimes x^{(2)},\> &e &\otimes x^{(3)} &\} \>&=&\> h &\otimes x
\notag \\
&&&\{e \otimes x,\> &h &\otimes x^{(3)} &\} \>&=&\> f &\otimes 1
\notag \\
&&&\{e \otimes x^{(2)},\> &h &\otimes x^{(3)} &\} \>&=&\> f &\otimes x
\label{eq-15} \\
&&&\{e \otimes x^{(3)},\> &h &\otimes x^{(3)} &\} \>&=&\> f &\otimes x^{(2)}
\notag \\
&\{e \otimes x,\> &f \otimes x^{(3)} \} \>=\>\>
&\{e \otimes x^{(3)},\> &f &\otimes x &\} \>&=&\> g &\otimes 1
\notag \\
&\{e \otimes x^{(2)},\> &f \otimes x^{(3)} \} \>=\>\>
&\{e \otimes x^{(3)},\> &f &\otimes x^{(2)} &\} \>&=&\> g &\otimes x
\notag \\
&&&\{h \otimes x,\> &h &\otimes x^{(3)} &\} \>&=&\> g &\otimes 1
\notag \\
&&&\{h \otimes x^{(2)},\> &h &\otimes x^{(3)} &\} \>&=&\> g &\otimes x
\notag \\
&&&\{e \otimes 1,\> &\partial& &\} \>&=&\> f &\otimes \beta x^{(3)}
\notag
\end{alignat}

A tedious, but straightforward computation shows that the following $3$ 
elements:
\begin{align*}
&h \otimes 1 + e \otimes x + (e \otimes x)^{[2]} ,       \\
&h \otimes (1 + x^{(2)}) + e \otimes x^{(3)} ,   \\
&h \otimes 1 + (e \otimes x^{(2)})^{[2]} + \delta (h \otimes x^{(3)})^{[2]}
\end{align*}
(note that $h \otimes x^{(2)} = (e \otimes x^{(3)})^{[2]}$, so the first and
the second lines here are transposed with respect to the substitution 
$x \leftrightarrow x^{(3)}$) linearly span a torus in the $p$-envelope of $\mathscr L$, so the absolute toral rank of $\mathscr L$ is at least $3$, 
a contradiction.
\end{proof}

\section{On $15$-dimensional simple Lie algebras}\label{sec-15}

The algebras $\mathscr L$ defined by multiplication table (\ref{eq-15}) are 
simple. Indeed, assume notation from the proof of Theorem \ref{th-main}, and let
$I$ be a nonzero ideal of $\mathscr L$. Any filtration on $\mathscr L$ induces 
a filtration on $I$, whose associated graded algebra is an ideal in the 
semisimple algebra $L$. Hence, $I$ contains $\ess \otimes \mathcal O_1(2)$, the
socle of $L$,  as a vector space. The multiplication table (\ref{eq-15}) reveals that then $I$ contains $g \otimes \langle 1,x \rangle$ and the element
$g \otimes \delta x + \partial$, and hence coincides with $\mathscr L$.

In \cite[Example at pp.~691--692]{skryabin}, a series of simple Lie algebras
satisfying the conditions of Theorem \ref{th-skr}, is constructed. These 
algebras are filtered deformations of semisimple Lie algebras of the form 
$$
W_1^\prime(n) \otimes \mathcal O_1(2) + e_{2^n-2} \otimes \langle 1,x \rangle 
+ K\partial .
$$ 
The smallest algebra in the series, obtained when $n=2$, coincides with algebra
(\ref{eq-15}) for $\beta = \delta = 0$.

At present, we do not know whether algebras (\ref{eq-15}) are isomorphic or not
for different values of parameters $\beta$ and $\delta$. Computer calculations
suggest that all these algebras share many numerical invariants: they are of
absolute toral rank $3$, the second cohomology $\Homol^2(\mathscr L,K)$ 
vanishes, the cohomology $\Homol^1(\mathscr L,\mathscr L)$, 
$\Homol^2(\mathscr L,\mathscr L)$, and $\Homol^3(\mathscr L,K)$ is of dimension
$4$, $13$, and $15$ respectively, they do not possess nontrivial symmetric invariant bilinear forms, the $p$-envelope coincides with the whole 
derivation algebra, and hence is of dimension $19$, and the subalgebra generated by absolute zero divisors (i.e., elements $x\in \mathscr L$
such that $(\ad x)^2 = 0$) has dimension $7$.

In \cite{eick}, a computer-generated list of simple Lie algebras over 
$\mathsf{GF}(2)$ of small dimensions is presented. It seems that the algebras 
considered here are not isomorphic to any $15$-dimensional algebra from this 
list, though, again, at present a rigorous proof of this is lacking. A thorough
study of the known $15$-dimensional simple Lie algebras is deferred to another 
paper.

\section*{Acknowledgements}

Thanks are due to Askar Dzhumadil'daev, Dimitry Leites, Serge Skryabin, and
the anonymous referee for useful remarks and clarifications. GAP \cite{gap} was
utilized to check some of the computations performed in this paper. Grishkov was
supported by Russian Science Foundation (project 16-11-10002) (\S\S 1--4), and 
by CNPq (grant 308221/2012-5) and FAPESP (grant 14/09310-5) (\S\S 5--6). 
Zusmanovich was supported by FAPESP (grant 13/12050-2), by the Statutory City of
Ostrava (grant 0924/2016/Sa\v{S}), and by the Ministry of Education and Science of the Republic of Kazakhstan 
(grant 0828/GF4). Most of this work was carried out during Zusmanovich's visit
to the University of S\~ao Paulo.

\end{document}